\documentclass[a4paper, oneside]{article}
\usepackage{enumitem} %abc-listojen tekemistä varten
\usepackage{lmodern} %tarvitaan tekstin näkymiseksi oikein pdf-tiedostossa
\usepackage{gastex} %automaattien kuvia varten
\usepackage{amssymb,amsthm,amsmath} %ams
\usepackage{bm}
\usepackage[ansinew]{inputenc} %skandien syöttö (windows XP), kommentoi väärä pois
\usepackage{verbatim}
\usepackage{tikz}
\usepackage{calligra}		% For calligraphic r, part 1
\usepackage{authblk}
\usetikzlibrary{arrows}
\usetikzlibrary{decorations.pathreplacing}
\usetikzlibrary{automata}
\usetikzlibrary{positioning}

\theoremstyle{definition}
\newtheorem{definition}{Definition}[section]
\newtheorem{lemma}[definition]{Lemma}
\newtheorem{proposition}[definition]{Proposition}
\newtheorem*{theorem*}{Theorem}
\newtheorem{theorem}[definition]{Theorem}
\newtheorem{example}[definition]{Example}
\newtheorem{corollary}[definition]{Corollary}
\newtheorem{remark}[definition]{Remark}
\newtheorem{problem}[definition]{Problem}
\newtheorem*{problem*}{Problem}

\newcommand{\N}{\mathbb{N}}
\newcommand{\Npos}{{\mathbb{N}_+}}

\newcommand{\Z}{\mathbb{Z}}

\newcommand{\graph}[1]{{\mathcal{#1}}} % Graafi
\newcommand{\glue}{\otimes} % Gluing configurations
\newcommand{\abs}[1]{\vert #1 \vert} % Itseisarvo
\newcommand{\orb}[1]{\mathcal{O}(#1)} % Orbit
\newcommand{\gf}{\mathrm{GF}} % Glider fleet
\newcommand{\gl}{\boxed{\leftarrow}} % Left glider
\newcommand{\gr}{\boxed{\rightarrow}} % Right glider
\newcommand{\z}{\bm{0}} % Non-constant zero-configuration
\newcommand{\one}{\bm{1}} % Non-constant one-configuration
\newcommand{\yl}{{\uparrow}} % Nuoli ylös
\newcommand{\al}{{\downarrow}} % Nuoli alas
\newcommand{\arr}{\ensuremath{\mathcalligra{r}}} % For calligraphic r, part 4
\newcommand{\syn}{{\mathrm{S}}} % Syntactic monoid class
\newcommand{\cont}{{\mathrm{C}}} % Contexts
\newcommand{\grp}[1]{{\mathcal{#1}}} % Group
\newcommand{\digs}{\Sigma} % Symbol set consisting of digits

\DeclareMathOperator{\cyl}{Cyl} % Cylinder set
\DeclareMathOperator{\aut}{Aut} % Automorphism
\DeclareMathOperator{\occ}{occ} % Occurrences
\DeclareMathOperator{\id}{Id} % Identity map

\DeclareMathAlphabet{\mathcalligra}{T1}{calligra}{m}{n}		% For calligraphic r, part 2
\DeclareFontShape{T1}{calligra}{m}{n}{<->s*[2.2]callig15}{} % For calligraphic r, part 3

\begin{document}

\title{Glider automata on all transitive sofic shifts}

\author{Johan Kopra}

\affil{Department of Mathematics and Statistics, \\FI-20014 University of Turku, Finland}
\affil{jtjkop@utu.fi}

\date{}

\maketitle

\setcounter{page}{1}

\begin{abstract}
For any infinite transitive sofic shift $X$ we construct a reversible cellular automaton (i.e. an automorphism of the shift $X$) which breaks any given finite point of the subshift into a finite collection of gliders traveling into opposing directions. This shows in addition that every infinite transitive sofic shift has a reversible CA which is sensitive with respect to all directions. As another application we prove a finitary Ryan's theorem: the automorphism group $\aut(X)$ contains a two-element subset whose centralizer consists only of shift maps. We also show that in the class of $S$-gap shifts these results do not extend beyond the sofic case.
\end{abstract}

\providecommand{\keywords}[1]{\textbf{Keywords:} #1}
\noindent\keywords{sofic shifts, synchronizing shifts, cellular automata, automorphisms}

\section{Introduction}

Let $X\subseteq A^\Z$ be a one-dimensional subshift over a symbol set $A$. If $w$ is a finite word over $A$, we may say that an element $x\in X$ is $w$-finite if it begins and ends with infinite repetitions of $w$ (in particular the bi-infinite repetition $w^\Z$ is $w$-finite). In this paper we consider the problem of constructing reversible cellular automata (CA) on $X$ which decompose all $w$-finite configurations into collections of gliders traveling into opposing directions. As a concrete example, consider the binary full shift $X=\{0,1\}^\Z$ and the map $G=P_3\circ P_2\circ P_1:X\to X$ defined as follows. In any $x\in X$, $P_1$ replaces every occurrence of $0010$ by $0110$ and vice versa, $P_2$ replaces every occurrence of $0100$ by $0110$ and vice versa, and $P_3$ replaces every occurrence of $00101$ by $00111$ and vice versa. In Figure \ref{std} we have plotted the sequences $x,G(x),G^2(x),\dots$ on consecutive rows for some $0$-finite $x\in X$. It can be seen that the sequence $x$ eventually diffuses into two different ``fleets'', the one consisting of 1s going to the left and the one consisting of 11s going to the right. It can be proved, along similar lines as in the proofs of Lemma~\ref{left} and Lemma~\ref{right}, that this diffusion happens eventually no matter which finite initial point $x\in X$ is chosen (this is also Theorem~5.1.5 of~\cite{KopDiss}). In Section~\ref{glidertransSect} we construct, on all infinite transitive sofic shifts $X$, a function $G_X$ that we call a diffusive glider CA and that has the same diffusion property as the CA $G$ above. The essential statement is contained in Theorem~\ref{absSofic}. This generalizes the similar construction done for mixing subshifts of finite type (SFT) in~\cite{Kop19b}.

\begin{figure}[ht]
\centering
\includegraphics{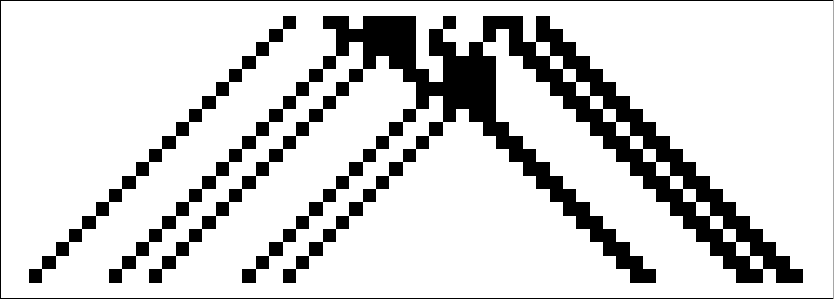}
\caption{The diffusion of $x\in X$ under the map $G:X\to X$. White and black squares correspond to digits $0$ and $1$ respectively.}
\label{std}
\end{figure}

We will perform most parts of the construction of the diffusive glider CA $G_X$ in the more general framework of synchronizing subshifts. One reason for this is that the statements and proofs of the auxiliary lemmas become simpler without using the extra structure of soficness. Using this framework in the construction also allows us to formulate a degree of freedom in the choice of the word $w$ with respect to which we consider finiteness: whenever $X$ is a transitive sofic shift and $w^\Z\in X$ contains a synchronizing word, then we can construct a $G_X$ which diffuses all $w$-finite points. This is relevant, because in SFTs every sufficiently long word is synchronizing, but this is no longer true for sofic shifts. In Subsection~\ref{choiceSubSect} we show that the construction of the CA $G_X$ can fail in an essential way if we do not require that $w^\Z$ contains a synchronizing word.

The existence of such a diffusive glider CA $G_X$ on a subshift $X$ is interesting for several reasons, the first being that $G_X$ can be used to convert an arbitrary finite $x\in X$ into another sequence $G_X^t(x)$ (for some $t\in\Npos$) with a simpler structure, which nevertheless contains all the information concerning the original point $x$ because $G_X$ is invertible. Such maps have been successfully applied to other problems. For example, the paper~\cite{Salo19} contains a construction of a finitely generated group $\grp{G}$ of reversible CA on $A^\Z$ (when $\abs{A}=4$) whose elements can implement any permutation on any finite collection of $0$-finite non-constant configurations that belong to different shift orbits. An essential part of the construction is that one of the generators of $\grp{G}$ is a diffusive glider CA on $A^\Z$. Another example is the construction of a physically universal cellular automaton $G$ on $A^\Z$ (when $\abs{A}=16$) in $\cite{ST17}$. Also here it is essential that $G$ is a diffusive glider CA (but $G$ also implements certain additional collision rules for gliders).

The CA $G_X$ is also interesting when considered in the framework of directional dynamics introduced by Sablik in~\cite{Sab08}. The map $G_X$ shows that every infinite transitive sofic shift $X$ has a reversible CA which is sensitive with respect to all directions. This result is in some sense the best possible, which can be seen by considering a natural class of synchronizing subshifts known as $S$-gap shifts. We show that an $S$-gap shift $X_S$ has a reversible CA which is sensitive with respect to all directions if and only if $X_S$ is an infinite transitive sofic shift.

We also consider a finitary version of Ryan's theorem. Let $X$ be a subshift and denote the set of its reversible CA by $\aut(X)$, which we may consider as an abstract group. According to Ryan's theorem~\cite{BLR88,Ryan74} the center of the group $\aut(X)$ is generated by the shift map $\sigma$ if $X$ is a transitive SFT. There may also be subsets $S\subseteq\aut(X)$ whose centralizers $C(S)$ are generated by $\sigma$. Denote the minimal cardinality of such a finite set $S$ by $k(X)$. In \cite{Salo19} it was proved that $k(X)\leq 10$ when $X$ is the full shift over the four-letter alphabet. In the same paper it is noted that $k(X)$ is an isomorphism invariant of $\aut(X)$ and therefore computing it could theoretically separate $\aut(X)$ and $\aut(Y)$ for some mixing SFTs $X$ and $Y$. Finding good isomorphism invariants of $\aut(X)$ is of great interest, and it is an open problem whether for example $\aut(\{0,1\}^\Z)\simeq\aut(\{0,1,2\}^\Z)$ (Problem 22.1 in \cite{Boy08}). We show that $k(X)=2$ for all infinite transitive sofic shifts, the proof of which uses our diffusive glider automorphism construction. In contrast, we can show that $k(X_S)=\infty$ whenever $X_S$ is a non-sofic $S$-gap shift.

This paper largely follows Chapter~5 of the author's PhD thesis~\cite{KopDiss}, where the construction of $G_X$ was done for infinite mixing sofic shifts $X$.

\section{Preliminaries}
\label{prelim}
In this section we recall some preliminaries concerning symbolic dynamics. The book~\cite{LM95} is a standard reference to the topic.

A finite set $A$ containing at least two elements (\emph{letters}) is called an \emph{alphabet} and the set $A^\Z$ of bi-infinite sequences (\emph{configurations}) over $A$ is called a \emph{full shift}. Formally any $x\in A^\Z$ is a function $\Z\to A$ and the value of $x$ at $i\in\Z$ is denoted by $x[i]$. It contains finite and one-directionally infinite subsequences denoted by $x[i,j]=x[i]x[i+1]\cdots x[j]$, $x[i,\infty]=x[i]x[i+1]\cdots$ and $x[-\infty,i]=\cdots x[i-1]x[i]$. Occasionally we signify the symbol at position zero in a configuration $x$ by a dot as follows:
\[x=\cdots x[-2]x[-1].x[0]x[1]x[2]\cdots.\]

A configuration $x\in A^\Z$ is \emph{periodic} if there is a $p\in\Npos$ such that $x[i+p]=x[i]$ for all $i\in\Z$. Then we may also say that $x$ is $p$-periodic or that $x$ has period $p$.

A \emph{subword} of $x\in A^\Z$ is any finite sequence $x[i,j]$ where $i,j\in\Z$, and we interpret the sequence to be empty if $j<i$. Any finite sequence $w=w[1] w[2]\cdots w[n]$ (also the empty sequence, which is denoted by $\lambda$) where $w[i]\in A$ is a \emph{word} over $A$. The concatenation of a word or a left-infinite sequence $u$ with a word or a right-infinite sequence $v$ is denoted by $uv$. A word $u$ is a \emph{prefix} of a word or a right-infinite sequence $x$ if there is a word or a right-infinite sequence $v$ such that $x=uv$. Similarly, $u$ is a \emph{suffix} of a word or a left-infinite sequence $x$ if there is a word or a left-infinite sequence $v$ such that $x=vu$. The set of all words over $A$ is denoted by $A^*$, and the set of non-empty words is $A^+=A^*\setminus\{\lambda\}$. The set of words of length $n$ is denoted by $A^n$. For a word $w\in A^*$, $\abs{w}$ denotes its length, i.e. $\abs{w}=n\iff w\in A^n$. For any word $w\in A^+$ we denote by ${}^\infty w$ and $w^\infty$ the left- and right-infinite sequences obtained by infinite repetitions of the word $w$. We denote by $w^\Z\in A^\Z$ the configuration defined by $w^\Z[in,(i+1)n-1]=w$ (where $n=\abs{w}$) for every $i\in\Z$. We say that $x\in A^\Z$ is $w\emph{-finite}$ if $x[-\infty,i]={}^\infty w$ and $x[j,\infty]=w^\infty$ for some $i,j\in\Z$.

Any collection of words $L\subseteq A^*$ is called a \emph{language}. For any $S\subseteq A^\Z$ the collection of words appearing as subwords of elements of $S$ is the \emph{language} of $S$, denoted by $L(S)$. For any $L,K\subseteq A^*$, let
\[LK=\{uv\mid u\in L, v\in K\}, \qquad L^*=\{w_1\cdots w_n\mid n\geq 0,w_i\in L\}\subseteq A^*.\]
If $\lambda\notin L$, define $L^+=L^*\setminus\{\lambda\}$ and if $\epsilon\in L$, define $L^+=L^*$.

Given $x\in A^\Z$ and $w\in A^+$ we define the sets of left (resp. right) occurrences of $w$ in $x$ by
\begin{flalign*}
&\occ_\ell(x,w)=\{i\in\Z\mid x[i,i+\abs{w}-1]=w\} \\
\mbox{(resp.)}\quad &\occ_\arr(x,w)=\{i\in\Z\mid x[i-\abs{w}+1,i]=w\}.
\end{flalign*}
Note that both of these sets contain the same information up to a shift in the sense that $\occ_\arr(x,w)=\occ_\ell(x,w)+\abs{w}-1$. Typically we refer to the left occurrences and we say that $w\in A^n$ \emph{occurs} in $x\in A^\Z$ at position $i$ if $i\in\occ_\ell(x,w)$.

For $x,y\in A^\Z$ and $i\in\Z$ we denote by $x\glue_i y\in A^\Z$ the ``gluing'' of $x$ and $y$ at $i$, i.e. $(x\glue_i y)[-\infty,i-1]=x[-\infty,i-1]$ and $(x\glue_i y)[i,\infty]=y[i,\infty]$. Typically we perform gluings at the origin and we denote $x\glue y=x\glue_0y$.

We define the shift map $\sigma_A:A^\Z\to A^\Z$ by $\sigma_A(x)[i]=x[i+1]$ for $x\in A^\Z$, $i\in\Z$. The subscript $A$ in $\sigma_A$ is typically omitted. The set $A^\Z$ is endowed with the product topology (with respect to the discrete topology on $A$), under which $\sigma$ is a homeomorphism on $A^\Z$.  Any closed set $X\subseteq A^\Z$ such that $\sigma(X)=X$ is called a \emph{subshift}. The restriction of $\sigma$ to $X$ may be denoted by $\sigma_X$, but typically the subscript $X$ is omitted. The \emph{orbit} of a point $x\in X$ is $\orb{x}=\{\sigma^i(x)\mid i\in\Z\}$. Any $w\in L(X)\setminus{\epsilon}$ and $i\in\Z$ determine a \emph{cylinder} of $X$
\[\cyl_X(w,i)=\{x\in X\mid w \mbox{ occurs in }x\mbox{ at position }i\}.\]

Next we define the classes of subshifts considered in this paper.

\begin{definition}
A subshift $X$ is \emph{transitive} if for all words $u,v\in L(X)$ there is $w\in L(X)$ such that $uwv\in L(X)$.
\end{definition}

\begin{definition}
A subshift $X$ is \emph{sofic} if $L(X)$ is a regular language.
\end{definition}

Alternatively, any sofic subshift can be given as the collection of labels of bi-infinite paths on some finite labeled directed graph. Yet another characterization can be given by considering syntactic monoids.

\begin{definition}Let $X$ be any subshift. The set of contexts of $w\in L(X)$ is defined by $\cont_X(w)=\{(w_1,w_2)\mid w_1ww_2\in L(X)\}$. We define an equivalence relation called the \emph{syntactic relation} on $L(X)$ as follows. For any $u,v\in L(X)$ let $u\sim v$ if $\cont_X(u)=\cont_X(v)$. The equivalence class containing $w\in L(X)$ is denoted by $\syn_X(w)$ and the collection of all equivalence classes is denoted by $\syn_X$. The subscript $X$ can be omitted when the subshift is clear from the context. By adjoining a zero element $0$ to $\syn_X$ we get a \emph{syntactic monoid} where multiplication is defined by $\syn_X(u)\syn_X(v)=\syn_X(uv)$ if $uv\in L(X)$, and otherwise the product of two elements is equal to $0$. It is easy to show that this monoid operation is well defined.
\end{definition}

It is known that a subshift $X$ is sofic if and only if $\syn_X$ is finite, see e.g. Theorem 6.1.2 in~\cite{Kit97}.

\begin{definition}
Given a subshift $X$, we say that a word $w\in L(X)$ is synchronizing if
\[\forall u,v\in L(X): uw,wv\in L(X)\implies uwv\in L(X).\]
We say that a transitive subshift $X$ is synchronizing if $L(X)$ contains a synchronizing word.
\end{definition}
Transitive sofic shifts in particular are synchronizing, which follows by using the results of~\cite{LM95} in Section~3.3 and in Exercise~3.3.3.

Next we define the structure preserving transformations between subshifts.

\begin{definition}
Let $X\subseteq A^\Z$ and $Y\subseteq B^\Z$ be subshifts. We say that the map $F:X\to Y$ is a \emph{morphism} from $X$ to $Y$ if there exist integers $m\leq a$ (memory and anticipation) and a \emph{local rule} $f:A^{a-m+1}\to B$ such that $F(x)[i]=f(x[i+m],\dots,x[i],\dots,x[i+a])$. The quantity $d=a-m$ is the \emph{diameter} of the local rule $f$. If $X=Y$, we say that $F$ is a \emph{cellular automaton} (CA). If we can choose $f$ so that $-m=a=r\geq 0$, we say that $F$ is a radius-$r$ CA.
\end{definition}

By Hedlund's theorem~\cite{Hed69} a map $F:X\to Y$ is a morphism if and only if it is continuous and $F\circ\sigma=\sigma\circ F$. We say that subshifts $X\subseteq A^\Z$ and $Y\subseteq B^\Z$ are \emph{conjugate} if there is a bijective morphism (a conjugacy) $F:X\to Y$. Bijective CA are called either \emph{reversible CA} or \emph{automorphisms}. The set of all automorphisms of $X$ is a group denoted by $\aut(X)$.

\begin{remark}
Technically it does not make any difference whether an element $F\in \aut(X)$ is called a reversible CA or an automorphism. In this paper we will make a distinction based on the \emph{role} the map $F$ plays in a given context. If we think of $F$ as forming a dynamical system, i.e. we are interested in repeated iteration of the map $F$ on the points of $X$, then we say that $F$ is a cellular automaton. If on the other hand it is natural to think of $F$ as an element of $\aut(X)$, e.g. if we are interested in the totality of the action of some larger group $\grp{G}\subseteq\aut(X)$ containing $F$, then we say that $F$ is an automorphism. Sometimes this distinction is a bit blurry.
\end{remark}

The notions of almost equicontinuity and sensitivity can be defined for general topological dynamical systems. We omit the topological definitions, because for cellular automata on transitive subshifts there are combinatorial characterizations for these notions using blocking words.

\begin{definition}
Let $F:X\to X$ be a radius-$r$ CA and $w\in L(X)$. We say that $w$ is a \emph{blocking word} if there is an integer $e$ with $\abs{w}\geq e\geq r+1$ and an integer $p\in[0,\abs{w}-e]$ such that
\[\forall x,y\in\cyl_X(w,0), \forall n\in\N, F^n(x)[p,p+e-1]=F^n(y)[p,p+e-1].\]
\end{definition}

The following is proved in Proposition~2.1 of~\cite{Sab08}.

\begin{proposition}\label{equiblock}
If $X$ is a transitive subshift and $F:X\to X$ is a CA, then $F$ is almost equicontinuous if and only if it has a blocking word.
\end{proposition}

We say that a CA on a transitive subshift is \emph{sensitive} if it is not almost equicontinuous. The notion of sensitivity is refined by Sablik's framework of directional dynamics~\cite{Sab08}.

\begin{definition}
Let $F:X\to X$ be a cellular automaton and let $p,q\in\Z$ be coprime integers, $q>0$. Then $p/q$ is a \emph{sensitive direction} of $F$ if $\sigma^p\circ F^q$ is sensitive. Similarly, $p/q$ is an \emph{almost equicontinuous direction} of $F$ if $\sigma^p\circ F^q$ is almost equicontinuous.
\end{definition}

This definition is best understood via the \emph{space-time diagram} of $x\in X$ with respect to $F$, in which successive iterations $F^t(x)$ are drawn on consecutive rows (see Figure~\ref{shift} for a typical space-time diagram of a configuration with respect to the shift map). By definition $-1=(-1)/1$ is an almost equicontinuous direction of $\sigma:A^\Z\to A^\Z$ because $\sigma^{-1}\circ \sigma=\id$ is almost equicontinuous. This is directly visible in the space-time diagram of Figure~\ref{shift}, because it looks like the space-time diagram of the identity map when it is followed along the dashed line. Note that the slope of the dashed line is equal to $-1$ with respect to the vertical axis extending downwards in the diagram.

\begin{figure}[ht]
\centering
\includegraphics{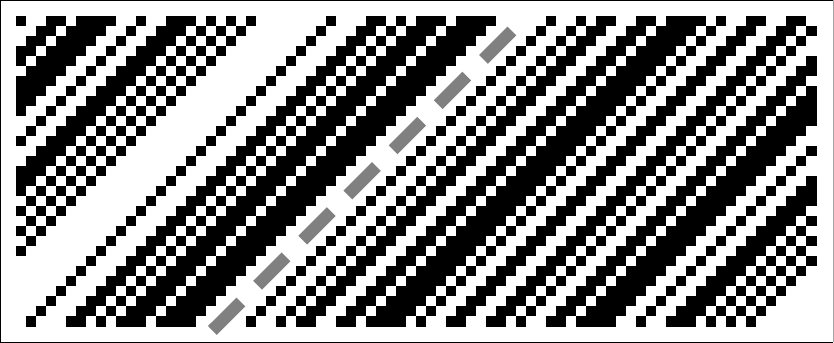}
\caption{A space-time diagram of the binary shift map $\sigma$. White and black squares correspond to digits $0$ and $1$ respectively. The dashed line shows an almost equicontinuous direction.}
\label{shift}
\end{figure}

\section{Markers on synchronizing subshifts}\label{markerSect}

In this section we find a collection of marker words of suitable form in any infinite synchronizing subshift. The precise result is stated in Propositions~\ref{0prim} and~\ref{01} and it may be of independent interest. The markers with good properties are found by transforming the subshift via a sequence of conjugacies through multiple lemmas.

\begin{lemma}\label{synborder}Let $X$ be a subshift and $u,v\in L(X)$ synchronizing words. If $w_1,w_2\in L(X)$ are words both of which have $u$ as a prefix and $v$ as a suffix, then $\syn_X(w_1)=\syn_X(w_2)$.\end{lemma}
\begin{proof}Let $t_1,t_2\in L(X)$ be such that $t_1w_1t_2\in L(X)$. In particular $t_1u\in L(X)$ and by assumption $w_2\in L(X)$, so by using the fact that $u$ is synchronizing it follows that $t_1w_2\in L(X)$. We also know that $vt_2\in L(X)$, so by using the fact that $v$ is synchronizing it follows that $t_1w_2t_2\in L(X)$. By symmetry, from $t_1w_2t_2\in L(X)$ it would follow that $t_1w_1t_2\in L(X)$, which proves the lemma.\end{proof}

\begin{definition}
Given a subshift $X\subseteq A^\Z$, we say that $w\in L(X)$ has a unique successor in $X$ (resp. a unique predecessor) if $wa\in L(X)$ (resp. $aw\in L(X)$) for a unique $a\in A$. Then we say that $a$ is the unique successor (resp. the unique predecessor) of $w$.\end{definition}

\begin{definition}
Let $X\subseteq A^\Z$ be a subshift and let $w=w_1\cdots w_n\in L(X)$ with all $w_i\in\ A$ distinct. If $w_i$ have unique successors for $1\leq i<n$, we say that $w$ is future deterministic in $X$ and if $w_j$ have unique predecessors for $1<j\leq n$, we say that $w$ is past deterministic in $X$. If $w$ is both future and past deterministic in $X$, we say that $w$ is deterministic in $X$.\end{definition}

Determinism of $w$ means that any symbol $a$ occurring in $w$ can occur in $x\in X$ only within an occurrence of $w$. If $X$ is infinite and transitive, then it is easy to see that the determinism of $w$ implies that all the symbols of $w$ are distinct.

\begin{lemma}\label{primeconj}
Let $X\subseteq A^\Z$ be a subshift and let $A'=\{a'\mid a\in A\}$. If $\psi:X\to X'\subseteq(A\cup A')^\Z$ is a surjective morphism and for all $x\in X$, $i\in \Z$, $a\in A$ it holds that $\psi(x)[i]\in\{a,a'\}\implies x[i]=a$ (i.e. $\psi$ does nothing else in configurations than add some primes as superscripts), then $\psi$ is a conjugacy. Furthermore, let $w=w_1\cdots w_n\in L(X)\cap L(X')$ and $w'=w_1'\cdots w_n'$. Then also the following hold.
\begin{itemize}
\item Assume that $w_i w_{i+1}',w_i' w_{i+1}\notin L(X')$ for $1\leq i<n$. If $w$ is future (resp. past) deterministic in $X$, then $w$ is future (resp. past) deterministic also in $X'$. 
\item Assume that $w$ is a synchronizing word for $X$ which is \emph{blocking} with respect to $\psi$ in the sense that for all $x,y\in\cyl_X(w,0)$,
\begin{flalign*}
& x[0,\infty]=y[0,\infty]\implies \psi(x)[n,\infty]=\psi(y)[n,\infty] \mbox{ and }\\
& x[-\infty,n-1]=y[-\infty,n-1]\implies \psi(x)[-\infty,-1]=\psi(y)[-\infty,-1]
\end{flalign*}
and whose \emph{priming is determined} either \emph{from the right} or \emph{from the left} in the sense that either
\begin{flalign*}
\forall x,y\in\cyl_X(w,0):\mbox{ }&x[0,\infty]=y[0,\infty] \\
&\implies \psi(x)[0,n-1]=\psi(y)[0,n-1] \mbox{ or }\\
\forall x,y\in\cyl_X(w,0):\mbox{ }&x[-\infty,n-1]=y[-\infty,n-1] \\
&\implies \psi(x)[0,n-1]=\psi(y)[0,n-1]
\end{flalign*}
respectively. Then $w$ is a synchronizing word for $X'$.
\end{itemize}

\end{lemma}
\begin{proof}
To see that $\psi$ is a conjugacy it suffices to show that $\psi$ is injective, but this is obvious.

Now assume that $w$ satisfies the assumption in the first item and that $w$ is future deterministic in $X$. We show that $w$ is future deterministic in $X'$. To see that $w_i$ ($1\leq i< n$) has a unique successor in $X'$, let $x\in X$ be such that $\psi(x)[0]=w_i$. Then also $x[0]=w_i$ and since $w$ is future deterministic in $X$ it follows that $x[0,1]=w_i w_{i+1}$ and $\psi(x)[0,1]\in\{w_iw_{i+1},w_iw_{i+1}'\}$. Since by assumption $w_i w_{i+1}'\notin L(X')$, it follows that $\psi(x)[0,1]=w_iw_{i+1}$ and $w_{i+1}$ is the unique successor of $w_i$ in $X'$. The proof for past determinism is symmetric.

Now assume that $w$ is a synchronizing word which is blocking and whose priming is determined from the right. Assume that $x_1',x_2'\in X'$ both have an occurrence of $w$ at the origin. To see that $w$ is a synchronizing word of $X'$, we need to show that $x_1'\glue x_2'$ (the gluing of $x_1'$ and $x_2'$ at the origin) belongs to $X'$. Let therefore $x_1,x_2\in X$ be such that $\psi(x_i)=x_i'$, so in particular both $x_i$ have an occurrence of $w$ at the origin. Since $w$ is synchronizing in $X$ it follows that $y=x_1\glue x_2\in X$. In other words $y\in\cyl_X(w,0)$, $x_1[-\infty,n-1]=y[-\infty,n-1]$ and $x_2[0,\infty]=y[0,\infty]$. Since $y$ is blocking, it follows that $\psi(y)[n,\infty]=\psi(x_2)[n,\infty]=x_2'[n,\infty]$ and $\psi(y)[-\infty,-1]=\psi(x_1)[-\infty,-1]=x_1'[-\infty,-1]$. Since the priming of $w$ is determined from the right, it follows that $\psi(y)[0,n-1]=\psi(x_2)[0,n-1]=x_2'[0,n-1]=w$. In total, $x_1'\glue x_2'=\psi(y)\in X'$. The case where the priming of $w$ is determined from the left is similar.\end{proof}

\begin{lemma}
Let $X\subseteq A^\Z$ be a subshift and let $A'=\{a'\mid a\in A\}$. Given $w=w_1\cdots w_n\in L(X)$ with all $w_i\in A$ distinct there is a conjugacy $\psi:X\to X'\subseteq (A\cup A')^\Z$ such that $w\in L(X')$ and $w$ is future deterministic in $X'$. Moreover, if $w^\Z\in X$ then $w^\Z\in X'$, and if $w$ is a synchronizing word of $X$ then $w$ is a synchronizing word of $X'$.
\end{lemma}
\begin{proof}
Let $\psi:X\to (A\cup A')^\Z$ be a morphism defined by
\[
\psi(x)[i]=
\begin{cases}
x[i]' & \mbox{ when } x[i]=w_j \mbox{ and } x[i,i+n-j]\neq w_j w_{j+1}\cdots w_n \\
& \mbox{ for some }1\leq j<n, \\
x[i] &\mbox{ otherwise.}
\end{cases}
\]
By Lemma \ref{primeconj} $\psi$ induces a conjugacy between $X$ and $X'=\psi(X)$. If $x\in X$ contains an occurrence of $w$ at the origin, then $\psi(x)$ also contains an occurrence of $w$ at the origin and $w\in L(X')$. If $w^\Z\in X$, we can here choose $x=w^\Z$ to show that $w^\Z\in X'$. To see that $w_i$ ($1\leq i< n$) has a unique successor in $X'$, assume to the contrary that $w_i a\in L(X')$ for some $a\in (A\cup A')\setminus\{w_{i+1}\}$. Then in particular there is $x\in X$ such that $w_i a$ occurs in $\psi(x)$ at position $0$. But then by definition of $\psi$, $x[0,n-i]=w_i w_{i+1}\cdots w_n$ and $\psi(x)$ contains an occurrence of $w_i w_{i+1}$ at the origin, contradicting the choice of $a$. If $w$ is a synchronizing word of $X$, then from the second item of Lemma \ref{primeconj} it follows that $w$ is a synchronizing word of $X'$ (the priming of $w$ is determined both from the left and from the right).
\end{proof}

\begin{lemma}
Let $X\subseteq A^\Z$ be a subshift and let $A'=\{a'\mid a\in A\}$. Let also $w=w_1\cdots w_n\in L(X)$ with all $w_i\in A$ distinct be such that $w$ is future deterministic in $X$. Then there is a conjugacy $\psi:X\to X'\subseteq (A\cup A')^\Z$ such that $w\in L(X')$ and $w$ is deterministic in $X'$. Moreover, if $w^\Z\in X$ then $w^\Z$ in $X'$, and if $w$ is a synchronizing word of $X$ then $w$ is a synchronizing word of $X'$.
\end{lemma}
\begin{proof}
Let $\psi:X\to (A\cup A')^\Z$ be a morphism defined by
\[
\psi(x)[i]=
\begin{cases}
x[i]' & \mbox{ when } x[i]=w_j \mbox{ and } x[i-j+1,i]\neq w_1 w_2\cdots w_j \\
& \mbox{ for some }1<j\leq n, \\
x[i] & \mbox{ otherwise.}
\end{cases}
\]
By Lemma \ref{primeconj} $\psi$ induces a conjugacy between $X$ and $X'=\psi(X)$. If $x\in X$ contains an occurrence of $w$ at the origin, then $\psi(x)$ also contains an occurrence of $w$ at the origin and $w\in L(X')$. If $w^\Z\in X$, we can here choose $x=w^\Z$ to show that $w^\Z\in X'$. The first item in Lemma \ref{primeconj} applies to show that $w$ is future deterministic in $X'$, and the same argument as in the proof of the previous lemma shows that $w$ is past deterministic. If $w$ is a synchronizing word of $X$, then from the second item of Lemma \ref{primeconj} it follows that $w$ is a synchronizing word of $X'$ (the priming of $w$ is determined both from the left and from the right).
\end{proof}

\begin{lemma}
Let $X\subseteq A^\Z$ be a subshift and let $w=w_1\cdots w_n\in L(X)$ with all $w_i$ distinct. There is an alphabet $B\supseteq A$ and a subshift $X'\subseteq B^\Z$ which is conjugate to $X$ such that $w\in L(X')$ and $w$ is deterministic in $X'$. Moreover, if $w^\Z\in X$ then $w^\Z\in X'$, and if $w$ is a synchronizing word of $X$ then it is also a synchronizing word of $X'$.
\end{lemma}
\begin{proof}
This follows by applying the two previous lemmas.
\end{proof}

\begin{definition}
The $n$-th higher power shift $X^{[n]}$ of a subshift $X\subseteq A^\Z$ is the image of $X$ under the map $\beta_n(x):X\to (A^n)^\Z$ defined by $\beta_n(x)[i]=x[i-k,i-k+n-1]$ (where $k=\lfloor n/2\rfloor$) for all $x\in X$, $i\in\N$. All higher power shifts are conjugate to the original subshift.
\end{definition}

We are now ready to present the main propositions of this section.

\begin{proposition}\label{0prim}
Let $X\subseteq A^\Z$ be a synchronizing subshift and let $\z\in L(X)$ be such that $\z^\Z\in X$, the minimal period of $\z^\Z$ is $\abs{\z}$ and $\z^k$ is synchronizing for some $k\in\Npos$. Up to recoding to a conjugate subshift we may assume that $\z$ is deterministic and synchronizing and that all symbols of $\z$ are distinct.
\end{proposition}
\begin{proof}
Denote $\z=0_1\cdots 0_p$ ($0_i\in A$, $p\in\Npos$). For sufficiently large $n$, $\beta_n(\z^\Z)[0,\abs{\z}-1]$ has all symbols distinct and $\beta_n(\z^\Z)[0,k\abs{\z}-1]=\beta_n(\z^\Z)[0,\abs{\z}-1]^k$ is a synchronizing word of $X^{[n]}$, so up to conjugacy we may assume that the symbols of $\z$ are distinct. By the previous lemma we may assume up to conjugacy that $\z$ is deterministic in $X$.

Let $\psi:X\to (A\cup A')^\Z$ be a morphism defined by
\[
\psi(x)[i]=
\begin{cases}
x[i]' & \mbox{ when } x[i]=0_j \mbox{ and } x[i,i+(p-j)+(k-1)p]\neq 0_j \cdots 0_p\z^{k-1}, \\
x[i] & \mbox{ otherwise.}
\end{cases}
\]
By Lemma \ref{primeconj} $\psi$ induces a conjugacy between $X$ and $X'=\psi(X)$. Clearly $\z^\Z=\psi(\z^\Z)\in X'$, and by Lemma \ref{primeconj} the word $\z$ is deterministic in $X'$. To see that $\z$ is synchronizing, assume that $x_1',x_2'\in X'$ both have an occurrence of $\z$ at the origin. We need to show that $x_1'\glue x_2'$ (the gluing of $x_1'$ and $x_2'$ at the origin) belongs to $X'$. Let therefore $x_1,x_2\in X$ be such that $\psi(x_i)=x_i'$, so in particular both $x_i$ have an occurrence of $\z^k$ at the origin. Since $\z^k$ is synchronizing in $X$ it follows that $y=x_1\glue x_2\in X$, and clearly $x_1'\glue x_2'=\psi(y)\in X'$.
\end{proof}

\begin{proposition}\label{01}
Let $X\subseteq A^\Z$ be an infinite synchronizing subshift and let $\z\in L(X)$ be such that $\z^\Z\in X$, $\z$ is deterministic and synchronizing and all symbols of $\z$ are distinct. Up to recoding to a conjugate subshift we may assume there is a word $\one\in L(X)$, $\abs{\one}\geq 2$, such that $\z$ and $\one$ satisfy the following:
\begin{itemize}
\item $\z^\Z\in X$, $\z$ is deterministic and synchronizing and all symbols of $\z$ are distinct
\item none of the symbols of $\z$ occur in $\one$
\item $\z\one^*\z\subseteq L(X)$
\item $\abs{\one}\equiv K\pmod {\abs{\z}}$ where $K=\gcd(\abs{\z},\abs{\one})$
\item if $w\in L(X)$ is such that $\z w\z\in L(X)$, then $K$ divides $\abs{w}$.
\end{itemize}
\end{proposition}
\begin{proof}
For any $X'$ that is conjugate to $X$ and that satisfies the first item it is possible to define the quantity 
\[K(X')=\min\{\gcd(\abs{\z},\abs{w})\mid w,\z w\z\in L(X)\setminus\{\lambda\}, w\notin A^*\z A^*\}.\]
Without loss of generality (up to conjugacy) we may assume in the following that $K(X)=\min_{X'}K(X')$.

There is some $w\in L(X)\setminus\{\lambda\}$ such that $\z w\z\in L(X)$, $\z$ is not a subword of $w$ and $\gcd(\abs{\z},\abs{w})=K(X)$. In the following we fix some such word $w\in L(X)$.

Denote $\z=0_1\cdots 0_p$ ($0_i\in A$, $p\in\Npos$). Let $A'=\{a'\mid a\in A\}$ and let $\psi:X\to (A\cup A')^\Z$ be a morphism defined by
\[
\psi(x)[i]=
\begin{cases}
x[i]' & \mbox{ when } x[i]=0_j \mbox{ and } x[i-j-\abs{w}+1,i]= w0_1 0_2\cdots 0_j \\
x[i] & \mbox{ otherwise.}
\end{cases}
\]
By Lemma \ref{primeconj} $\psi$ induces a conjugacy between $X$ and $X'=\psi(X)$. Clearly $\z^\Z=\psi(\z^\Z)\in X'$, and by Lemma \ref{primeconj} the word $\z$ is synchronizing and deterministic in $X'$ (the priming of $\z$ is determined from the left). Now denote $\z'=0_1'\cdots 0_p'$, let $u=w\z$ and $\one'=w\z'$. It directly follows that $\abs{\one'}\geq 2$ and that none of the symbols of $\z$ occur in $\one'$. Because ${}^\infty\z\one'\z^\infty=\psi({}^\infty\z u\z^\infty)\in X'$, we have $\z\one'\z\in L(X')$ and $K(X')\leq\gcd(\abs{\z},\abs{\one'})=\gcd(\abs{\z},\abs{w})=K(X)\leq K(X')$, where the last inequality follows because $X$ was chosen so that $K(X)$ is minimal. Therefore $\gcd(\abs{\z},\abs{\one'})=K(X')$. By choosing $\one=\one'^k$ for a suitable $k\in\Npos$ we can also get $\gcd(\abs{\z},\abs{\one})=K(X')$ and $\abs{\one}\equiv K(X')\pmod{\abs{\z}}$. Since $\z w\z\in L(X)$ and $\z$ is synchronizing in $X$, it follows that ${}^\infty\z (u^k)^*\z^\infty\subseteq X$, and by applying $\psi$ to these points it follows that $\z\one^*\z\subseteq L(X')$. We may therefore assume in the following that $X$ satisfies the first four items and that $K=K(X)=\min_{X'}K(X')$.

To see that the last item holds, assume to the contrary that there exists  $v\in L(X)$ such that $\z v\z\in L(X)$ and $\abs{v}=nK+r$ for some $n\in\N$, $0<r<K$. We may assume without loss of generality (by considering some suitable subword of $v$ instead if necessary) that none of the symbols of $\z$ occur in $v$. We may also write $\abs{\z}=n_1 K$ and $\abs{\one}=n_2\abs{\z} + K$. Let $\psi:X\to (A\cup A')^\Z$ be a morphism defined by
\[
\psi(x)[i]=
\begin{cases}
x[i]' & \mbox{ when } x[i]=0_j \mbox{ and } x[i-j-\abs{\z v}+1,i]= \z v0_1 0_2\cdots 0_j \\
x[i] & \mbox{ otherwise.}
\end{cases}
\]
By Lemma \ref{primeconj} $\psi$ induces a conjugacy between $X$ and $X'=\psi(X)$. Clearly $\z^\Z=\psi(\z^\Z)\in X'$, and by Lemma \ref{primeconj} the word $\z$ is synchronizing and deterministic in $X'$ (the priming of $\z$ is determined from the left). By choosing $k\in\N$ such that $n+k$ is divisible by $n_1$ and by denoting $u=v\z'\one^k$ we see that ${}^\infty\z u\z^\infty=\psi({}^\infty\z v\z\one^k\z^\infty)\in X'$ and $\z u\z\in L(X')$ (note that $\one^k\neq v$ because $v$ is not divisible by $K$) but
\begin{flalign*}
&\gcd(\abs{\z},\abs{u})=\gcd(n_1 K,(nK+r)+(n_1 K)+k(n_2 n_1 K+K)) \\
&=\gcd(n_1 K,(n+k)K+r)=\gcd(n_1 K,r)<K=K(X),
\end{flalign*}
contradicting $K(X)=\min_{X'}K(X')$.
\end{proof}

We will use the special words in the statement of the previous proposition in conjunction with Lemma~\ref{markerauto}, which explicitly states the principle that we will use to construct reversible CA in the following sections. This principle is known as the marker method and it has been stated in different sources with varying levels of generality, e.g. for full shifts in \cite{Hed69} and for mixing SFTs in \cite{BLR88}. The statement requires the notion of an overlap.

\begin{definition}Let $u,v\in A^*$. We say that $w\in A^*$ is an \emph{overlap} of $u$ and $v$ if $w$ is a suffix of $u$ and a prefix of $v$, or if $w=u$ is a subword of $v$, or if $w=v$ is a subword of $u$. We say that $w$ is a trivial overlap if $w=\epsilon$ or $w=u=v$.\end{definition}

\begin{lemma}\label{markerauto}Let $X$ be a subshift, let $u\in L(X)$ and let $W$ be a finite collection of words such that $uWu\subseteq L(X)$ and each pair of (not necessarily distinct) elements of $uWu$ has only $u$ as an overlap in addition to the trivial ones. Let $\pi:uWu\to uWu$ be a permutation that preserves the lengths and syntactic relation classes of elements of $uWu$. Then there is a reversible CA $F:X\to X$ such that for any $x\in X$ the point $F(x)$ is gotten by replacing every occurrence of any element $w\in uWu$ in $x$ by $\pi(w)$.\end{lemma}
\begin{proof}The map $F$ is well defined since the elements of $uWu$ can overlap nontrivially only by $u$. For the same reason elements of $uWu$ occur in $F(x)$ at precisely the same positions than in $x$, and then the reversibility of $F$ follows from the reversibility of $\pi$. To see that $F(X)\subseteq X$, note first that replacing a single occurrence of a word $uwu\in uWu$ in $x\in X$ by $\pi(uwu)$ yields another configuration from $X$, because by assumption $uwu$ and $\pi(uwu)$ are in syntactic relation. Then an induction shows that after making any finite number of such replacements the resulting point is still contained in $X$. From this $F(x)\in X$ follows by compactness.\end{proof}

\section{Constructing glider CA on (sofic) synchronizing shifts}\label{glidertransSect}

In this section we will construct a cellular automaton $G_X$, whose most important properties are stated in the following theorem for easier reference. This essentially states that the behavior of Figure~\ref{std} can be replicated by reversible CA on all infinite transitive sofic shifts.

\begin{theorem}\label{absSofic}
Let $Y$ be an infinite transitive sofic subshift and let $\z^\Z\in Y$ be a periodic configuration containing a synchronizing word and whose minimal period is $\abs{\z}$. Then there is a conjugacy $\psi:Y\to X$ such that $\psi(\z^\Z)=\z^\Z\in X$, $\z$ is synchronizing and deterministic in $X$, and a reversible CA $G_X:X\to X$ such that there are
\begin{itemize}
\item words $\gl,\gr\in L(X)$ called \emph{left- and rightbound gliders},
\item languages of gliders $L_\ell=(\gl\z\z^*)^*\subseteq L(X)$ and $L_\arr=(\z^*\z\gr)^*\subseteq L(X)$ and
\item glider fleet sets $\gf_{\ell}={}^\infty\z L_\ell\z^\infty\subseteq X$ and $\gf_\arr={}^\infty\z L_\arr\z^\infty\subseteq X$ (note that in each element there are only finitely many occurrences of $\gl$ and $\gr$), whose elements are called \emph{glider fleets}
\end{itemize}
and for some $s\in\Npos$, which is a multiple of $\abs{\z}$, $G_X$ satisfies
\begin{itemize}
\item $G_X(x)=\sigma^{s}(x)$ for $x\in\gf_\ell$ and $G_X(x)=\sigma^{-s}(x)$ for $x\in\gf_\arr$ and
\item if $x\in X$ is a $\z$-finite configuration, then for every $N\in\N$ there exist $t,N_\ell,N_\arr,M\in\N$, $N_\ell,N_\arr\geq N$ such that $G_X^t(x)[-N_\ell,N_\arr]=\z^M$, $G_X^t(x)[-\infty,-(N_\ell+1)]\in{}^\infty\z L_\ell$ and $G_X^t(x)[N_\arr+1,\infty]\in L_\arr\z^\infty$.
\end{itemize}
\end{theorem}

We will see that almost all steps of the construction of $G_X$ work without the assumption of soficness. Therefore we are also able to construct a family of CA $G_{X,n}$ on not necessarily sofic $X$ which shares some of the functionality of the CA $G_X$. We will use the details of the construction of $G_X$ and $G_{X,n}$ in later sections. An alternative would be to include all the used properties in the statement of Theorem~\ref{absSofic}, but this would make the statement of the theorem significantly longer and less clear. We leave this modification as an exercise to the interested reader. 

To begin the construction, we start with an infinite synchronizing subshift $X\subseteq A^\Z$ and an arbitrary periodic configuration $\z^\Z\in X$ containing a synchronizing word. We will also assume in the rest of this section that there is a word $\one\in L(X)$ that together with $\z$ satisfies the statement of Proposition~\ref{01}: this can be done up to conjugacy by combining Propositions~\ref{0prim} and~\ref{01}.

Let $p=\abs{\z}$, $q=\abs{\one}$ and $K=\gcd(p,q)$. The words
\[\gl=\z^q \one \qquad \gr=\one^{p+1}\]
will be the left- and rightbound gliders. The languages of left- and rightbound gliders are
\[L_\ell=(\gl\z\z^*)^* \qquad L_\arr=(\z^*\z\gr)^*\]
and we define the glider fleet sets
\[\gf_{\ell}={}^\infty\z(\gl \z\z^*)^*\z^\infty \qquad \gf_\arr={}^\infty\z(\z^*\z \gr)^*\z^\infty.\]
These definitions cover the first three items in the statement of Theorem~\ref{absSofic}.

We now define reversible CA $P_1,P_2:X\to X$ as follows. In any $x\in X$,

\begin{itemize}
\item $P_1$ replaces every occurrence of $\z(\z^{q}\one)\z$ by $\z(\one^{p+1})\z$ and vice versa.
\item $P_2$ replaces every occurrence of $\z (\one^{p+1})\z$ by $\z(\one\z^q)\z$ and vice versa.
\end{itemize}

Each $P_i$ is defined as in Lemma \ref{markerauto} by $u=\z$, a set $B_i$ of two finite words and nontrivial permutations $\pi_i$. In each case the words in $u B_i u$ are of equal length and easily verified to have only trivial overlaps by Proposition~\ref{01}. By Lemma~\ref{synborder} both elements in each $u B_i u$ are in syntactic relation, so we conclude that Lemma \ref{markerauto} is applicable.

To define the CA $P_3$ let us assume in this paragraph that $X\subseteq A^\Z$ is a sofic shift, so $\syn_X$ is a finite set. If $\z=0_1\cdots 0_p$, denote $B=A\setminus\{0_1,\dots, 0_p\}$. Then also
\[P=\{\syn_X(\z w)\mid w\in L(X)\cap (B^K)^+,\z w\in L(X),\abs{w}>q(p+1)\}\]
is a finite set and we may choose a uniform $N_1\in \N$ such that for every $S\in P$ there is a word $w_S'\in L(X)\cap (B^K)^+$ with $S=\syn_X(\z w_S')$ and $q(p+1)<\abs{w_S'}\leq N_1$. The lengths of the words in $(\one\z)^+\one^+(\one^{p+1}\z)$ attain all sufficiently large multiples of $K$, so we can fix $N\in\N$ which is divisible by $K$ such that for every $S\in P$ there is a word $w_{S}\in (\one\z)^+\one^+(\one^{p+1}\z) w_S'$ of length $N$. Furthermore we assume that $N>\abs{\one^{p+1+p/K}}$ (this is needed in a later paragraph). In particular $\z w_S\in S$ by Lemma \ref{synborder}. Fix some such $w_S$, let $W_S'=\{w_{S,1},\dots, w_{S,k_S}\}$ be the set of those words from $L(X)\cap B^N$ such that $\z w_{S,i}\in S$ for $1\leq i\leq k_S$, denote $W_S=W_S'\cup\{w_S\}$ and $W=\bigcup_{S\in P}W_S$. For applying Lemma \ref{markerauto}, let $u=\epsilon$ and let $\pi:\z^{q+1} W\to \z^{q+1} W$ be the permutation that maps the elements of each $\z^{q+1} W_S$ cyclically, i.e. $\z^{q+1} w_S\to \z^{q+1} w_{S,1}\to\dots \to \z^{q+1} w_{S,k_S}\to \z^{q+1} w_S$. Define the reversible CA $P_3:X\to X$ that replaces occurrences of elements of $\z^{q+1} W_S$ using the permutation $\pi$.

For this paragraph fix some integer $n>\abs{\one^{p+1+p/K}}$. We define the CA $P_{4,n}$ that ``permutes words shorter than $n$ not containing $\z$'' as follows. For each $j\in\{1,\dots, p/K\}$ let $u_{j}'=\one\z^q\one^{j}$ (the names of all the words we define in this paragraph should contain the parameter $n$ in the index, but we suppress it to avoid clutter), and let $U_{j,n}'=\{u_{j,1}',\dots,u_{j,n_j}'\}\subseteq L(X)\cap B^+$ be the set of nonempty words of length at most $n-1$ such that $\z u_{j,i}'\z\in L(X)$, $u_{j,n_j}'=\one^{p+1+j}$ ($\abs{u_{j,n_j}'}<n$ by the choice of $n$), $\abs{u_{j,i}'}\equiv\abs{u_j'}\equiv(j+1)K\pmod p$, with the additional restriction that $\one,\one^{p+1}\notin U_{p/K,n}'$. Finally, these words are padded to constant length: let $u_j=\z^{c_j}u_j'$ and $u_{j,i}=\z^{c_{j,i}}u_{j,i}'$, where $c_j, c_{j,i}\geq q+1$ are chosen in such a way that all $u_j$, $u_{j,i}$ are of the same length for any fixed $j$. Let $U_{j,n}=\{u_j\}\cup\{u_{j,i}\mid 1\leq i\leq n_j\}$, $U_n=\bigcup_{j=1}^{p/K} U_{j,n}$. For applying Lemma \ref{markerauto}, let $u=\z$, let $V_{j,n},V_n\subseteq L(X)$ such that $\z V_{j,n}\z=U_{j,n}\z$, $\z V_n\z=U_n\z$ and let $\rho:\z V_n\z\to \z V_n\z$ be the permutation that maps the elements of each $\z V_{j,n}\z$ cyclically, i.e. $u_j\z\to u_{j,1}\z\to\dots  \to u_{j,n_j}\z\to u_j\z$. Define the reversible CA $P_{4,n}:X\to X$ that replaces occurrences of elements of $U_{j,n}\z$ using the permutation $\rho$.

In the case when $X$ is a sofic shift define $P_4=P_{4,N}$, where $N$ is the number defined two paragraphs above. In this case we can drop the subscript $N$ from the sets $U_{j,N}', U_{j,N}, U_N, V_{j,N}, V_N$ of the previous paragraph. 

The glider CA $G_{X,n}:X\to X$ (with parameter $n$) is defined as the composition $P_{4,n}\circ P_2\circ P_1$. If $X$ is sofic, the \emph{diffusive glider CA} $G_X:X\to X$ is defined as the composition $P_4\circ P_3\circ P_2\circ P_1$. All statements concerning the CA $G_X$ below contain the assumption that $X$ is sofic.

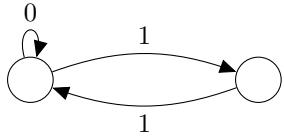
\begin{figure}[ht]
\centering
\begin{tikzpicture}[auto]
\node (l) at (0,0) [shape=circle,draw,minimum size=6mm] {};
\node (r) at (3,0)  [shape=circle,draw,minimum size=6mm] {};

\draw[-{triangle 45}] (l) to [out=20, in=160] node {$1$} (r);
\draw[-{triangle 45}] (r) to [out=200, in=340] node {$1$} (l);
\draw[-{triangle 45}] (l) to [in=75, out=105, loop above] node {$0$} ();
\end{tikzpicture}
\caption{The graph of the even shift.}
\label{egraph}
\end{figure}

\begin{example}
We will give the explicit construction of the diffusive glider CA $G_X:X\to X$ in the case when $X\subseteq \{0,1\}^\Z$ is the even shift containing those configurations in which no words from $\{01^{2n+1}0\mid n\in\N\}$ occur. More concretely, the configurations of $X$ are precisely the labels of all bi-infinite paths on the graph presented in Figure~\ref{egraph}. Let $\z=0$ and $\one=11$, so $p=\abs{\z}=1$, $q=\abs{\one}=2$ and $K=\gcd(\abs{\z},\abs{\one})=1$. It is easy to verify that these choices of $\z$ and $\one$ satisfy the statement of Proposition~\ref{01} (note in particular that the determinism of $\z$ is vacuously true because $\abs{\z}=1$). The CA $P_1$ replaces every occurrence of $000110$ by $011110$ and vice versa, $P_2$ replaces every occurrence of $011110$ by $011000$ and vice versa.

For defining the CA $P_3,P_4$, note that $B=\{0,1\}\setminus\{0\}=\{1\}$ (the set of symbols not in $\z$) and
\[P=\{\syn_X(0w)\mid w\in 1^+,\abs{w}>4\}=\{\syn_X(01^5),\syn_X(01^6)\}.\]
Denote $S_0=\syn_X(0)=\syn_X(01^6)$ and $S_1=\syn_X(01)=\syn_X(01^5)$ and choose $w_{S_0}'=111111$, $w_{S_1}'=11111$. Then we can choose
\begin{flalign*}
&w_{S_0}=110(11)^4 0w_{S_0}'=110111111110111111\mbox{ and } \\
&w_{S_1}=110110(11)^3 0w_{S_1}'=110110111111011111,
\end{flalign*}
which are of length $N=18$. If $w\in B^N$ then $w=1^{18}$ and $\syn_X(0w)=S_0$ and therefore $W_{S_0}'=\{w_{S_0,1}\}=\{1^{18}\}$, $W_{S_1}'=\emptyset$ and $P_3$ is the CA that replaces every occurrence of
\begin{flalign*}
&000w_{S_0}=000110111111110111111\mbox{ by } \\
&000w_{S_0,1}=000111111111111111111
\end{flalign*}
and vice versa.

Recall that $p=1$, so $u_j'$, $U_j'$, etc. need to be defined only for $j=1$. Let $u_1'=110011$ and $U_1'=\{u'_{1,i}\mid 1\leq i\leq 6\}$, where $u_{1,1}'=1^{16}$, $u_{1,2}'=1^{14}$, $u_{1,3}'=1^{12}$, $u_{1,4}'=1^{10}$, $u_{1,5}'=1^{8}$ and $u_{1,6}'=1^{6}$. These are padded to constant length: $u_1=0^{13} 110011$, $u_{1,1}=0^3 1^{16}$, $u_{1,2}=0^5 1^{14}$, $u_{1,3}= 0^7 1^{12}$, $u_{1,4}=0^9 1^{10}$, $u_{1,5}=0^{11} 1^{8}$ and $u_{1,6}=0^{13} 1^{6}$. are words of length $19$. The CA $P_4$ permutes occurrences of $0^{13} 1100110$, $0^3 1^{16}0$, $0^5 1^{14}0$, $0^7 1^{12}0$, $0^9 1^{10}0$, $0^{11} 1^{8}0$ and $0^{13} 1^{6}0$ cyclically.

The space-time diagram of a typical finite configuration $x\in X$ with respect to $G_X$ is plotted in Figure~\ref{soficstd}. In this figure it can be seen that $x$ eventually diffuses into two glider fleets, leaving the area around the origin empty.
\end{example}

\begin{figure}[ht]
\centering
\includegraphics[width = 4.44in, height = 1.03in]{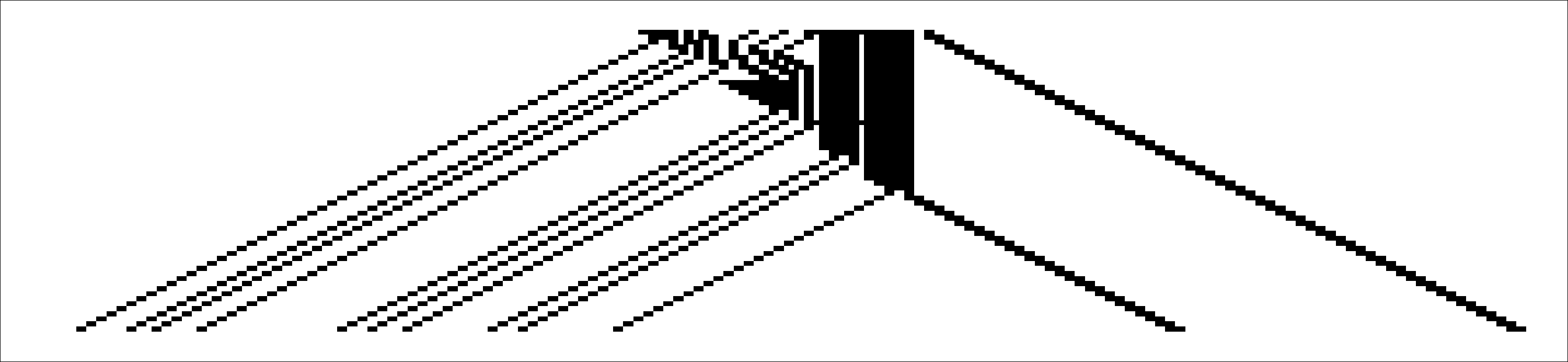}
\caption{Action of $G_X:X\to X$ on a typical $\z$-finite configuration of $X$ when $X$ is the even shift. White and black squares correspond to digits $0$ and $1$ respectively.}
\label{soficstd}
\end{figure}

Theorem~\ref{absSofic} predicts that the behavior observed in Figure~\ref{soficstd} also happens in general, thus giving justification for calling $G_X$ a diffusive glider CA. The following lemma covers the fourth item in Theorem~\ref{absSofic}.

\begin{lemma}\label{localshift}
If $x\in\gf_\ell$ (resp. $x\in\gf_\arr$), then $G_X(x)=G_{X,n}(x)=\sigma^{pq}(x)$ (resp. $G_X(x)=G_{X,n}(x)=\sigma^{-pq}(x))$.
\end{lemma}
\begin{proof}
We present the proof only for $G_X$. Assume that $x\in\gf_\ell$ (the proof for $x\in\gf_\arr$ is similar) and assume that $i\in\Z$ is some position in $x$ where $\gl$ occurs. Then
\begin{flalign*}
&x[i-p,i+(pq+q)+p-1]=\z\gl\z=\z(\z^q \one)\z \\
&P_1(x)[i-p,i+(pq+q)+p-1]=\z(\one^{p+1})\z \\
&P_2(P_1(x))[i-p-pq,i+q+p-1]=\z^q\z(\one\z)=\z\gl\z \\
&G_X(x)=P_4(P_3(P_2(P_1(x))))=P_2(P_1(x))),
\end{flalign*}
so every glider has shifted by distance $pq$ to the left and $G_X(x)=\sigma^{pq}(x)$.
\end{proof}

In fact, the previous lemma would hold even if $G_X$ and $G_{X,n}$ were replaced by $P_2\circ P_1$. The role of the part $P_4\circ P_3$ in $G_X$ for sofic $X$ is, for a given finite point $x\in X$, to ``erode'' non-$\z$ non-glider parts of $x$ from the left and to turn the eroded parts into new gliders. Similarly, for not necessarily sofic $X$, the part $P_{4,n}$ can erode non-$\z$ non-glider parts from the left, but in this case only under the assumption that these parts are shorter than $n$.  We will formalize this in a lemma, in the proof of which the following structural definitions will be useful.

\begin{definition}Let $n>\abs{\one^{p+1+p/K}}$. Assume that $x\notin\gf_\ell$ is a $\z$-finite element of $X$ not in $\orb{\z^\Z}$ and not containing occurrences of words from $B^n$. Then there is a maximal $i\in\Z$ such that
\[x[-\infty,i-1]\in{}^\infty\z L_\ell,\]
and there is a unique word $w\in \{\one\z\}\cup\{\one^{p+1}\z\}\cup (\bigcup_{j=1}^{p/K} U_{j,n}'\z)$ such that $w$ is a prefix of $x[i,\infty]$. Let $k=i+\abs{w}-1$. We say that $x$ is of $n$-\emph{left bound type} $(w,k)$ and that it has $n$-left bound $k$ (note that $k>i$).
\end{definition}

\begin{definition}Assume that $X$ is a sofic shift and that $x\notin\gf_\ell$ is a $\z$-finite element of $X$ not in $\orb{\z^\Z}$. Then there is a maximal $i\in\Z$ such that
\[x[-\infty,i-1]\in{}^\infty\z L_\ell,\]
and there is a unique word $w\in \{\one\z\}\cup\{\one^{p+1}\z\}\cup (\bigcup_{j=1}^{p/K} U_j'\z)\cup(\bigcup_{S\in P} W_S')$ such that $w$ is a prefix of $x[i,\infty]$. If $w\in\{\one\z,\one^{p+1}\z\}$ or $w\in U_j'\z$, let $k=i+\abs{w}-1$ and otherwise let $k=i+\abs{\one\z}-1$. We say that $x$ is of \emph{left bound type} $(w,k)$ and that it has left bound $k$ (note that $k>i$).
\end{definition}

We outline a deterministic method to narrow down the word $w$ in the definition of left bound type in a way that clarifies its existence and uniqueness (the case of $n$-left bound type would be similar). First, by the maximality of $i$ it follows that $x[i]\in B$. If $x[i,i+N-1]\in B^N$, then $w\in W_{\syn_X(\z x[i,i+N-1])}'$ directly by the definition of the sets $W_S'$. Otherwise $x[i,i+N-1]\notin B^N$ and there is a minimal $m<N$ such that $x[i,i+m-1]\in B^m$ and $x[i+m,i+m+p-1]=\z$. Then $\z x[i,i+m-1]\z\in L(X)$ and by the last item of Proposition~\ref{01} $m$ is divisible by $K$. Then by the second to last item of Proposition~\ref{01} $w\in U_j'\z$ for some $j\in\{1,\dots,p/K\}$ \emph{unless} we have specifically excluded $x[i,i+m-1]$ from all the sets $U_j'$. But this happens precisely if $x[i,i+m-1]\in\{\one,\one^{p+1}\}$, in which case $w\in \{\one\z,\one^{p+1}\z\}$.

The point of this definition is that if $x$ is of left bound type $(w,k)$, then the CA $G_X$ and $G_{X,n}$ will create a new leftbound glider at position $k$ and break it off from the rest of the configuration.

\begin{lemma}\label{left}Assume that $x\in X$ has left bound $k$. Then there exists $t\in\Npos$ such that the left bound of $G_X^t(x)$ is strictly greater than $k$. Moreover, the left bound of $G_X^{t'}(x)$ is at least $k$ for all $t'\in\N$.\end{lemma}
\begin{proof}
Let $x\in X$ be of left bound type $(w,k)$ with $w\in \{\one\z\}\cup\{\one^{p+1}\z\}\cup (\bigcup_{j=1}^{p/K} U_j'\z)\cup(\bigcup_{S\in P} W_S')$. The gliders to the left of the occurrence of $w$ near $k$ move to the left at constant speed $pq$ under the action of $G_X$ without being affected by the remaining part of the configuration.

\begin{description}
\item[Case 1.] Assume that $w=\one^{p+1}\z$. Then $P_1(x)[k-(q+2p)+1,k]=\z\one\z$ and we proceed to Case 4. 
\item[Case 2.] Assume that $w=\one\z$. Then $x[k-(q+2)p-q+1,k]\neq\z(\z^q\one)\z=\z\gl\z$ because otherwise the left bound of $x$ would already be greater than $k$, so $P_1(x)[k-2p-q+1,k]=\z\one\z$ and we proceed to Case 4.
\item[Case 3.] Assume that $w=u_{j,i}'\z$ for $1\leq j\leq p/K$, $1\leq i\leq n_j$. There is a minimal $t\in\N$ such that $P_3(P_2(P_1(G_X^t(x))))[k-(p+\abs{u_j})+1,k]=u_{j,i}\z$. Denote $y=G_X^{t+n_j-i+1}(x)$ so in particular $y[k-(p+\abs{u_j})+1,k]=u_j\z$. If $j>1$, then $y$ is of left bound type $(u_{j-1,i'},k)$ for some $1\leq i'<n_{j-1}$ and we may repeat the argument in this paragraph with a smaller value of $j$. If $j=1$, then $P_1(x)[k-(q+2p)+1,k]=\z\one\z$ and we proceed as in Case 4.
\item[Case 4.] Assume that $P_1(x)[k-(q+2p)+1,k]=\z\one\z$. If $P_1(x)[k-(q+2p)+1,k+qp]=\z(\one\z^q)\z$, then $G_X(x)[k-(q+2p)+1,k+qp]=P_2(P_1(x))[k-(q+2p)+1,k+qp]=\z\one^{p+1}\z$, $G_X(x)$ is of left bound type $(\one^{p+1}\z, k+qp)$ and we are done. Otherwise $P_2(P_1(x))[k-(q+2p)+1,k]=\z\one\z$. Denote $y=P_3(P_2(P_1(x)))$. If $y[k-(q+2p)+1,k]\neq\z\one\z$, then $G_X(x)=P_4(y)$ is of left bound type $(w_{S,1},k)$ for some $S\in P$ and we proceed as in Case 5. Otherwise $y[k-(q+2p)+1,k]=\z\one\z$. If $G_X(x)[k-(q+2p)+1,k]=P_4(y)[k-(q+2p)+1,k]\neq \z\one\z$, then $G_X(x)$ is of left bound type $(u_{j,1},k')$ for some $1\leq j\leq p/K$, $k'>k$ and we are done. Otherwise $G_X(x)[-\infty,k]\in {}^\infty\z L_\ell$, the left bound of $G_X(x)$ is strictly greater than $k$ and we are done.
\item[Case 5.] Assume that $w=w_{S,i}$ for $S\in P$ and $1\leq i\leq k_S$. Then there is a minimal $t\in\N$ such that $G_X^t(x)[k-\abs{\one\z}+1,\infty]$ has prefix $w_S$. Since $w_S$ has prefix $\one\z$, it follows that $G_X^t(x)[-\infty,k]\in{}^\infty\z L_\ell$. Thus the left bound of $G_X^t(x)$ is strictly greater than $k$ and we are done.
\end{description}
\end{proof}

The same method can be used to prove the following lemma in the not necessarily sofic case, but this time Case 5 of the previous proof does not come into play.

\begin{lemma}\label{leftsynch}Let $n>\abs{\one^{p+1+p/K}}$ and assume that $x\in X$ has $n$-left bound $k$. Then there exists $t\in\Npos$ such that the $n$-left bound of $G_{X,n}^t(x)$ is strictly greater than $k$. Moreover, the $n$-left bound of $G_{X,n}^{t'}(x)$ is at least $k$ for all $t'\in\N$.\end{lemma}

For the right bounds we have a simpler definition.

\begin{definition}
If $x\notin\gf_\arr$ is a non-zero finite element of $X$, then there is a minimal $k\in\Z$ such that
\[x[k+1,\infty]\in L_\arr \z^\infty\]
and we say that $x$ has right bound $k$.
\end{definition}

\begin{lemma}\label{right}Assume that $x\in X$ has right bound $k$. Then there exists $t\in\Npos$ such that the right bound of $G_X^t(x)$ is strictly less than $k$. Moreover, the right bound of $G_X^{t'}(x)$ is at most $k$ for all $t'\in\N$.\end{lemma}
\begin{proof}
Let us assume to the contrary that the right bound of $G_X^t(x)$ is at least $k$ for every $t\in\Npos$.

Assume first that the right bound of $G_X^t(x)$ is equal to $k$ for every $t\in\Npos$. By Lemma~\ref{left} the left bound of $G_X^t(x)$ is arbitrarily large for suitable choice of $t\in\Npos$, which means that for some $t\in\Npos$ $G_X^t(x)$ contains only $\gl$-gliders to the left of $k+3pq$ and only $\gr$-gliders to the right of $k$. This can happen only if $G_X^t(x)[k+1,k+3pq-1]$ does not contain any glider of either type. Then the right bound of $G_X^{t+1}(x)$ is at most $k-pq$, a contradiction.

Assume then that the right bound of $G_X^t(x)$ is strictly greater than $k$ for some $t\in\Npos$ and fix the minimal such $t$. This can happen only if $P_1(G_X^{t-1}(x))[k-(p+q)+1,k+(q+1)p]=\z\one\z^q\z$ and then $P_2(P_1(G_X^{t-1}(x)))[k-(p+q)+1,k+(q+1)p]=\z\one^{p+1}\z$. But neither $P_3$ nor $P_4$ can change occurrences of $\z\one^{p+1}\z$ in configurations (recall in particular that $\abs{w_S'}>\abs{\one^{p+1}}$ for all $S\in P$) so $G_X^t(x)[k-(p+q)+1,k+(q+1)p]=\z\one^{p+1}\z$. It follows that the right bound of $G_X^t(x)$ is at most $k-(p+q)$, a contradiction.
\end{proof}

Similarly one proves the following in the not necessarily sofic case.

\begin{lemma}\label{rightsynch}Let $n>\abs{\one^{p+1+p/K}}$, assume that $x\in X$ does not contain occurrences of words from $B^n$ and that $x$ has right bound $k$. Then there exists $t\in\Npos$ such that the right bound of $G_{X,n}^t(x)$ is strictly less than $k$. Moreover, the right bound of $G_{X,n}^{t'}(x)$ is at most $k$ for all $t'\in\N$.\end{lemma}

By inductively applying the previous lemmas we get the following pair of theorems. Theorem~\ref{glidertranssofic} covers the fifth item of Theorem~\ref{absSofic}, the last remaining part.

\begin{theorem}\label{glidertranssofic}If $x\in X$ is a finite configuration, then for every $N\in\N$ there exist $t,N_\ell,N_\arr,M\in\N$, $N_\ell,N_\arr\geq N$ such that $G_X^t(x)[-N_\ell,N_\arr]=\z^M$, $G_X^t(x)[-\infty,-(N_\ell+1)]\in{}^\infty\z L_\ell$ and $G_X^t(x)[N_\arr+1,\infty]\in L_\arr\z^\infty$.\end{theorem}

\begin{theorem}\label{glidersynch}Let $n>\abs{\one^{p+1+p/K}}$. If $x\in X$ is a finite configuration that does not contain occurrences of words from $B^n$, then for every $N\in\N$ there exist $t,N_\ell,N_\arr,M\in\N$, $N_\ell,N_\arr\geq N$ such that $G_{X,n}^t(x)[-N_\ell,N_\arr]=\z^M$, $G_{X,n}^t(x)[-\infty,-(N_\ell+1)]\in{}^\infty\z L_\ell$ and $G_{X,n}^t(x)[N_\arr+1,\infty]\in L_\arr\z^\infty$.\end{theorem}

Our construction proves the following theorem. Recall the notions of directional dynamics.

\begin{theorem}\label{sensitiveCA}
For every infinite transitive sofic shift $X$ there exists a reversible CA $F\in\aut(X)$ that has no almost equicontinuous directions.
\end{theorem}
\begin{proof}
We claim that $G_X:X\to X$ is such an automaton. To see this, assume to the contrary that there is an almost equicontinuous direction $r/s$ for coprime integers $r$ and $s$ such that $s>0$. This means that $F=\sigma^r\circ G_X^s$ is almost equicontinuous and admits a blocking word $w\in L(X)$. Since every word containing a blocking word is also blocking, we may choose $w$ so that $\z w\z\in L(X)$.

Assume first that $r\geq 0$. Define $x={}^\infty\z.w\z^\infty$ and $x_n={}^\infty\z.w\z^n\gl\z^\infty$ for all $n\in\Npos$. We claim that for some $n\in\Npos$ we can choose $t\in\N$ such that $F^t(x)[-\infty,-1]\neq F^t(x_n)[-\infty,-1]$, which would contradict $w$ being a blocking word. To see this, we apply Theorem~\ref{glidertranssofic} for some sufficiently large $N\in\N$ so that $G_X^t(x)[-N_\ell,N_\arr]=\z^M$, $G_X^t(x)[-\infty,-(N_\ell+1)]\in{}^\infty\z L_\ell$ and $G_X^t(x)[N_\arr+1,\infty]\in L_\arr\z^\infty$ for all $t$ larger than some $t_0\in\N$, where $N_\ell$, $N_\arr$ and $M$ are as in the statement of the theorem. Fix some $i\in\Npos$ such that $G_X^{t_0}(x)[\abs{w}+ip,\infty]=\z^\infty$ and for $j\in\Npos$ let $n_j=j+t_0 q$. Then $x_{n_j}={}^\infty\z.w\z^{j+t_0 q}\gl\z^\infty$ and by fixing $n=n_{i+k}$ for some sufficiently large $k\in\N$ we get $G_X^{t_0}(x_n)[N_\arr+1,\infty]\in L_\arr\z^*\z^k\gl\z^\infty$. It is possible to choose $t'\geq t_0$ so that $\occ_\arr(G_X^{t''}(x_n),\gl)\subseteq(-\infty,-1]$ for all $t''\geq t'$. Then it holds that $\abs{\occ_\arr(G_X^{t''}(x_n),\gl)}>\abs{\occ_\arr(G_X^{t''}(x),\gl)}$ for all $t''\geq t'$. Now let $t\in\N$ such that $st\geq t'$. Then $F^t(x_n)=\sigma^{rt}(G_X^{st}(x_n))$ and $F^t(x)=\sigma^{rt}(G_X^{st}(x))$, so $\abs{\occ_\arr(F^{t}(x_n),\gl)}>\abs{\occ_\arr(F^{t}(x),\gl)}$. Because we assumed that $r\geq 0$, it also follows that $\occ_\arr(F^{t}(x_n),\gl)\subseteq(-\infty,-1]$  and in particular $F^t(x)[-\infty,-1]\neq F^t(x_n)[-\infty,-1]$.

A symmetric argument yields a contradiction in the case $r\leq 0$.
\end{proof}

\begin{remark}\label{BoyleRmrk}
The assumption of $X$ being a sofic shift was used in the construction of $G_X$ only in the definition of the map $P_3$. To be more precise, we used the finiteness of the set 
\[P=\{\syn_X(\z w)\mid w\in L(X)\cap (B^K)^+,\z w\in L(X),\abs{w}>q(p+1)\}\]
and we noted that for this it is sufficient that $X$ is sofic. In fact it turns out that the soficness of $X$ is equivalent to $P$ being finite. To see the other direction, first note that if $P$ is finite then also $V=\{\syn_X(\z w)\mid w\in L(X),\z w\in L(X)\}$ is finite. As in~\cite{FF92}, we can construct a directed labeled graph called the \emph{Fischer cover} of $X$. This graph has the vertex set $V$ and an edge from $\syn_X(\z w)$ to $\syn(\z wa)$ with the label $a$ whenever $w\in A^*$, $a\in A$ and $\z wa\in L(X)$. By~\cite{FF92} the set $X'$ consisting of the labels of bi-infinite paths on this graph is dense in $X$. From the finiteness of the graph it follows that $X'$ is also compact, so $X=X'$ and $X$ is sofic.
\end{remark}

The assumption of soficness turns out to be even more essential in the context of the previous theorem. In Subsection~\ref{SpecSubSect} we will present a family of synchronizing subshifts on which it is impossible to carry out any construction analogous to that of $G_X$ in the sense that on these shifts the previous theorem does not hold.

\section{Implications related to Ryan's theorem}

In this section we discuss an application of the diffusive glider CA construction presented above to the study of the structure of the abstract group $\aut(X)$. The \emph{centralizer} of a set $S\subseteq\grp{G}$ (with respect to a group $\grp{G}$) is
\[C_{\grp{G}}(S)=\{g\in\grp{M}\mid g\circ h=h\circ g\mbox{ for every }h\in S\}.\]
In this section we consider centralizers with respect to some automorphism group $\aut(X)$ and we drop the subscript from the notation $C_{\aut(X)}(S)$. The subgroup generated by $S\subseteq\aut(X)$ is denoted by $\left<S\right>$. The following definition is by Salo from~\cite{Salo19}:

\begin{definition}\label{kDef}For a subshift $X$, let $k(X)\in\N\cup\{\infty,\bot\}$ be the minimal cardinality of a set $S\subseteq \aut(X)$ such that $C(S)=\left<\sigma\right>$ if such a set $S$ exists, and $k(X)=\bot$ otherwise.\end{definition}
It is a theorem of Ryan from~\cite{Ryan72} that $k(A^\Z)\neq\bot$, which he later generalized to $k(X)\neq\bot$ whenever $X$ is an infinite transitive SFT in~\cite{Ryan74}. This result is also presented in Theorem~7.7 of~\cite{BLR88} with an alternative proof. Section~7.6 of~\cite{Salo19} contains the following observation concerning the lower bounds of $k(X)$.
\begin{theorem}\label{soa}
Let $X$ be a subshift. The case $k(X)=0$ occurs if and only if $\aut(X)=\left<\sigma\right>$. The case $k(X)=1$ cannot occur.
\end{theorem}

For conjugate subshifts $X$ and $Y$ it necessarily holds that $k(X)=k(Y)$.

We will now show that $k(X)=2$ for all infinite transitive sofic shifts, the proof of which uses our diffusive glider CA construction and Lemma~\ref{mainlmm}. The lemma has been originally proved in~\cite{Kop19b}, and we will state it (and some associated definitions) in the generality needed.

\begin{definition}\label{gliderdef}
Given a subshift $X\subseteq A^\Z$, a \emph{diffusive glider automorphism group} is any tuple $(\grp{G},\z,\gl,\gr,s)$ (or just $\grp{G}$ when the rest of the tuple is clear from the context) where $\grp{G}\subseteq\aut(X)$ is a subgroup, $\z,\gl,\gr\in A^+$, $s\in\Npos$ and
\begin{itemize}
\item the sets $\gf_{\ell}={}^\infty\z(\gl \z\z^*)^*\z^\infty$ and $\gf_\arr={}^\infty\z(\z^*\z \gr)^*\z^\infty$ are characterized by
\begin{flalign*}
&\gf_{\ell}=\{x\in X\mid x\mbox{ is }\z\mbox{-finite and }G(x)=\sigma^s(x)\}\mbox{ and} \\
&\gf_{\arr}=\{x\in X\mid x\mbox{ is }\z\mbox{-finite and }G(x)=\sigma^{-s}(x)\}
\end{flalign*}
for some $G\in\grp{G}$
\item for every $x\in\gf_{\ell}$ it holds that $\abs{j-k}\geq\abs{\gl}$ whenever $j,k\in\occ_\ell(x,\gl)$ are distinct, i.e. the occurrences of $\gl$ do not overlap in any point of $\gf_{\ell}$ (and similarly for all $x\in\gf_{\arr}$)
\item for every $\z$-finite $x\in X$ and every $N\in\N$ there is a $G\in \grp{G}$ such that for every $i\in\Z$, $G(x)[i,i+N]\in L(\gf_{\ell})\cup L(\gf_{\arr})$.
\end{itemize}
If $\grp{G}$ is generated by a single automorphism $G\in\aut(X)$, we say that $G$ is a \emph{diffusive glider CA}.
\end{definition}

\begin{example}\label{glidertransxmpl}
Let $Y$ be an infinite transitive sofic shift. In the previous section we found a conjugate subshift $X$ on which we constructed the diffusive glider CA $G_X:X\to X$. We claim that this really is a diffusive glider CA in the sense of Definition~\ref{gliderdef} with an associated glider automorphism group $(\left<G_X\right>,\z,\gl,\gr,pq)$, where $p,q,\z,\gl,\gr$ and the fleets $\gf_\ell$ and $\gf_\arr$ are as in the previous section. 

By Lemma \ref{localshift} we know that for $i\in\{\ell,\arr\}$ and for $\delta(\ell)=1$, $\delta(\arr)=-1$
\[\gf_i\subseteq\{x\in X\mid x\mbox{ is }\z\mbox{-finite and }G_X(x)=\sigma^{\delta(i)pq}(x)\}\doteqdot S_i.\]
We prove the other inclusion when $i=\ell$, the case $i=\arr$ being similar. Assume therefore that $x\notin\gf_\ell$ is $\z$-finite and apply Theorem~\ref{glidertranssofic} for sufficiently large $M$. By Lemma~\ref{localshift} the set $\gf_\ell$ is invariant under the map $G_X$, so $G_X^t(x)\notin\gf_\ell$ and $G_X^t(x)$ contains an occurrence of $\gr$ which is shifted to the right by the map $G_X$. Therefore $G_X(G_X^t(x))\neq \sigma^{pq}(G_X^t(x))$ and $G_X^t(x)\notin S_\ell$. Since $S_\ell$ is invariant under the map $G_X$, it follows that $x\notin S_\ell$.

The second item in Definition~\ref{gliderdef} is clear and the third item follows by Theorem~\ref{glidertranssofic}.
\end{example}

We have a similar example on infinite synchronizing shifts.

\begin{example}\label{glidersynchxmpl}
Let $Y$ be an infinite synchronizing shift. In the previous section we found a conjugate subshift $X$ on which we constructed the glider CA $G_{X,n}:X\to X$ with parameter $n>\abs{\one^{p+1+p/K}}$. We claim that $(\left<\{G_{X,n}\mid n>\abs{\one^{p+1+p/K}}\}\right>,\z,\gl,\gr,pq)$ is a diffusive glider automorphism group, where $p,q,\z,\gl,\gr$ and the fleets $\gf_\ell$ and $\gf_\arr$ are as in the previous section.

Fix some $n>\abs{\one^{p+1+p/K}}$. By Lemma \ref{localshift} we know that for $i\in\{\ell,\arr\}$ and for $\delta(\ell)=1$, $\delta(\arr)=-1$
\[\gf_i\subseteq\{x\in X\mid x\mbox{ is }\z\mbox{-finite and }G_{X_n}(x)=\sigma^{\delta(i)pq}(x)\}\doteqdot S_i.\]
We prove the other inclusion when $i=\ell$, the case $i=\arr$ being similar. Assume therefore that $x\notin\gf_\ell$ is $\z$-finite. If $x$ contains no occurrences of words from $B^n$, we can use the same argument as in the previous example by using Theorem~\ref{glidersynch} instead of Theorem~\ref{glidertranssofic}. If on the other hand $x$ contains on occurrence of a word from $B^n$, let $k\in\Z$ be the maximal position at which such a word occurs. Then this word also occurs in $G_{X,n}(x)$ at position $k$, so $G_{X,n}(x)\neq\sigma^{pq}(x)$.

The second item in Definition~\ref{gliderdef} is clear. For the third item, let $x\in X$ be $\z$-finite and let $N\in\N$ be arbitrary. Fix some $n>\abs{\one^{p+1+p/K}}$ such that $x$ contains no occurrences of words from $B^n$. By Theorem~\ref{glidersynch} we can choose $t\in\N$ such that $G_{X,n}^t(x)[i,i+N]\in L(\gf_{\ell})\cup L(\gf_{\arr})$ for every $i\in\N$.
\end{example}

We also require the notion of an automorphism that fixes the orbit of a given periodic point in a given subshift.

\begin{definition}
For a subshift $X\subseteq A^\Z$ and a word $w\in A^+$ such that $w^\Z\in X$ denote $\aut(X,w)=\{F\in\aut(X)\mid F(\orb{w^\Z})=\orb{w^\Z}\}$.
\end{definition}

\begin{lemma}[\cite{Kop19b}, Lemma~1]\label{mainlmm}
Let $X\subseteq A^\Z$ be a subshift with a diffusive glider automorphism group $(\grp{G},\z,\gl,\gr,s)$ such that $\z$-finite configurations are dense in $X$. Assume that there is a strictly increasing sequence $(N_m)_{m\in\N}\in\N^\N$ and a sequence $(G_m)_{m\in\N}\in \grp{G}^\N$ such that for any $x\gr\z^\infty\in\gf_\arr$, ${}^\infty\z\gl y\in\gf_\ell$ we have
\begin{itemize}
\item $x\gr.\z^{N_m}\gl y\in X$
\item $G_m(x\gr.\z^N\gl y)=x\gr\z.\z^N\z\gl y$ for every $N>N_m$ such that $x\gr.\z^N\gr y\in X$ 
\item $G_m(x\gr.\z^{N_m}\gl y)=x\z\gr.\z^{N_m}\gl\z y$.
\end{itemize}
Then $C(\graph{G})\cap\aut(X,\z)=\left<\sigma\right>$.
\end{lemma}

As earlier, let $X$ be an infinite synchronizing shift of the form given in Proposition~\ref{01} and consider the notation of Section~\ref{glidertransSect}. First we define maps $F_1,F_2:X\to X$ as follows. In any $x\in X$,
\begin{itemize}
\item $F_1$ replaces every occurrence of $\z\gr\z\z\z\gl\z$ by $\z\gr\z\z\gl\z\z$ and vice versa
\item $F_2$ replaces every occurrence of $\z\gr\z\z\gl\z$ by $\z\z\gr\z\gl\z$ and vice versa.
\end{itemize}
It is easy to see that these maps are well-defined automorphisms of $X$. The automorphism $F:X\to X$ is then defined as the composition $F_2\circ F_1$. $F$ has the following properties. First, it replaces any occurrence of $\z\gr \z\z\z\gl \z$ by $\z\z\gr\z\gl\z\z$. Second, if $x\in X$ is a configuration containing only gliders $\gl$ and $\gr$ separated by words from $\z^+$ and if every occurrence of $\gl$ is sufficiently far from every occurrence of $\gr$, then $F(x)=x$.

\begin{proposition}\label{PropRyan}Let an infinite transitive sofic subshift $X\subseteq A^\Z$ and $G_X,F:X\to X$ be as above. Then $C(\left<G_X,F\right>)=\left<\sigma\right>$.\end{proposition}
\begin{proof}Let $(\left<G_X\right>,\z,\gl,\gr,pq)$ be the diffusive glider automorphism group from Example~\ref{glidertransxmpl}. If we define $\grp{G}=\left<G_X,F\right>$, then it directly follows that $(\grp{G},\z,\gl,\gr,pq)$ is also a diffusive glider automorphism group of $X$. We want to use Lemma~\ref{mainlmm} to show that $C(\grp{G})\cap\aut(X,\z)=\left<\sigma\right>$.

Recall that we denote $p=\abs{\z}$, $q=\abs{\one}$. Using the same notation as in the statement of Lemma~\ref{mainlmm}, let $(N_m)_{m\in\N}$ with $N_m=2mq+3$ and $(G_m)_{m\in\N}$ with $G_m=G_X^{-(m+1)}\circ F\circ G_X^m$. Let $x\gr\in{}^\infty\z L_r$, $\gl y\in L_\ell\z^\infty$ be arbitrary. Fix some $m\in\N$. Since $\z$ is synchronizing in $X$, it is clear that $x\gr.\z^{N_m}\gl y\in X$ and it is easy to verify that 
\begin{itemize}
\item $G_m(x\gr.\z^N\gl y)=x\gr\z.\z^N\z\gl y$ for $N>N_m$
\item $G_m(x\gr.\z^{N_m}\gl y)=x\z\gr.\z^{N_m}\gl\z y$.
\end{itemize}
Therefore $C(\grp{G})\cap\aut(X,\z)=\left<\sigma\right>$.

Now let $H\in C(\grp{G})$ be arbitrary. Let us show that $H\in\aut(X,\z)$. Namely, assume to the contrary that $H(\z^\Z)=w^\Z\notin\orb{\z^\Z}$ for some $w=w_1\cdots w_p$ ($w_i\in A$). The maps $P_k$ in the definition of $G_X$ have been defined so that $P_k(x)[i]=x[i]$ whenever $x$ contains occurrences of $\z$ only at positions strictly greater than $i$, so in particular $G_X(w^\Z)=w^\Z$. Consider $x={}^\infty\z.\gl\z^\infty\in\gf_\ell$ with the glider $\gl$ at the origin. Note that $H(x)[(i-1)p,ip-1]\neq w$ for some $i\in\Z$ (otherwise $H(x)=w^\Z=H(\z^\Z)$, contradicting the injectivity of $H$) and $H(x)[-\infty,ip-(jq)p-1]=\cdots www$ for some $j\in\Npos$. By the earlier observation on the maps $P_k$ it follows that $G_X^t(H(x))[-\infty, ip-(jq)p-1]=\cdots www$ for every $t\in\Z$ but $H(G_X^j(x))[ip-(j+1)qp,ip-(jq)p-1]=H(\sigma^{(pq)j}(x))[ip-(j+1)qp,ip-(jq)p-1]=H(x)[ip-qp,ip-1]\neq w^q$, contradicting the commutativity of $H$ and $G_X$. Thus $H\in\aut(X,\z)$.

We have shown that $H\in C(\grp{G})\cap\aut(X,\z)=\left<\sigma\right>$, so we are done.
\end{proof}

\begin{theorem}[Finitary Ryan's theorem]\label{transRyan}$k(X)=2$ for every infinite transitive sofic shift $X$.\end{theorem}
\begin{proof}Every nontrivial mixing sofic shift is conjugate to a subshift $X$ of the form given in Proposition~\ref{01}, so $k(X)\leq 2$ follows from the previous proposition. Clearly $\aut(X)\neq\left<\sigma\right>$, so by Theorem~\ref{soa} it is not possible that $k(X)<2$ and therefore $k(X)=2$.
\end{proof}

Ryan's result $k(X)\neq\bot$ can probably be generalized to synchronizing subshifts using the same type of argument as in~\cite{Ryan74}, but we have not seen this stated explicitly in print. We now outline an alternative proof in the glider CA framework.

\begin{proposition}\label{PropSynchRyan}Let an infinite synchronizing subshift $X\subseteq A^\Z$ and $G_{X,n},F:X\to X$ be as above. Then $C(\left<\{G_{X,n}\mid n>\abs{\one^{p+1+p/K}}\}\cup \{F\}\right>)=\left<\sigma\right>$.\end{proposition}
\begin{proof}
By using Example~\ref{glidertransxmpl} we see that $\left<\{G_{X,n}\mid n>\abs{\one^{p+1+p/K}}\}\cup \{F\}\right>$ is a diffusive glider automorphism group. Fix some $n>\abs{\one^{p+1+p/K}}$. We conclude by replacing every occurrence of $G_X$ by $G_{X,n}$ in the proof of Proposition~\ref{PropRyan}.
\end{proof}

Ryan's theorem immediately follows.

\begin{theorem}[Ryan's theorem]\label{synchRyan}$k(X)\neq\bot$ for every synchronizing subshift $X$.\end{theorem}

We end this section with the following remark. Finitary Ryan's theorem can be interpreted as a compactness result saying that, for a nontrivial mixing sofic shift $X$, the group $\aut(X)$ has a finite subset $S$ such that $C(S)=\left<\sigma\right>$. One may wonder whether this compactness phenomenon is more general: in Section~7.3 of~\cite{Salo19} the question was raised whether for a mixing SFT $X$ and for every $R\subseteq\aut(X)$ such that $C(R)=\left<\sigma\right>$ there is a finite subset $S\subseteq R$ such that also $C(S)=\left<\sigma\right>$. In the same section it was noted that to construct a counterexample it would be sufficient to find a locally finite group $\grp{G}\subseteq\aut(X)$ whose centralizer is generated by $\sigma$. A different strategy based on an ad hoc glider CA construction was used in~\cite{Kop19b} to construct a counterexample in the case when $X$ is the binary full shift. We are now in a position to easily generalize this counterexample to all infinite synchronizing subshifts by combining the following proposition with Proposition~\ref{PropSynchRyan}.

\begin{proposition}\label{noCompact}Let an infinite synchronizing subshift $X\subseteq A^\Z$ and $G_{X,n},F:X\to X$ be as above and let $S\subseteq\left<\{G_{X,n}\mid n>\abs{\one^{p+1+p/K}}\}\cup \{F\}\right>$ be finite. Then $C(S)\supsetneq\left<\sigma\right>$.\end{proposition}
\begin{proof}Assume to the contrary that $C(S)=\left<\sigma\right>$. Since $S$ is finite, it is easy to see that whenever $n\in\Npos$ is sufficiently large, the elements of $S$ cannot remove or add occurrences of the words $w_i=\z\one^{n+i}\z$ ($i\in\N$) in any configuration. Let therefore $H\in\aut(X)$ be the automorphism which given a point $x\in X$ replaces every occurrence of the pattern
\[0 w_3 0 w_1 0 w_2 0 \qquad \mbox{by} \qquad 0 w_3 0 w_2 0 w_1 0\]
and vice versa (it exists by Lemma~\ref{markerauto} with the choice $u=0$). The elements of $S$ cannot remove or add occurrences of the words defined above, so $H$ commutes with every element of $S$, a contradiction.
\end{proof}

\section{Restrictions to constructing glider automata}

\subsection{Example: the choice of $\z$ in mixing sofic shifts}\label{choiceSubSect}

In Section~\ref{glidertransSect} we constructed glider automata on an arbitrary infinite transitive sofic shift $X$ that can diffuse any $\z$-finite configuration into two glider fleets. In other words, the diffusion is guaranteed against the background of the periodic configuration $\z^\Z$, but in the construction we required that the word $\z$ satisfies the synchronization assumption of Proposition~\ref{01}. One may then ask whether this assumption is necessary. In particular, if we have a subshift $X\in A^\Z$ and a symbol $0\in A$ such that  $0^\Z\in X$, it would feel the most natural to consider finiteness with respect to this $1$-periodic configuration and ask whether there exists a reversible CA that can diffuse every $0$-finite configuration. We show by an example that sometimes this cannot be done.

In this subsection we consider the mixing sofic shift $X\subseteq\{0,1,a,b,\al,\yl\}^\Z$ whose language $L(X)$ consists of all the subwords of words in $L=(L_0 0^* L_1 0^*)^*$, where
\begin{flalign*}
&L_0=1(ab)^*\yl(ab)^*\al(ab)^*1\cup 1(ab)^*\al(ab)^*\yl(ab)^*1 \\
&L_1=1(ab)^*\al(ab)^*\al(ab)^*1\cup 1(ab)^*\yl(ab)^*\yl(ab)^*1.
\end{flalign*}
The intuition is that words $w_0\in L_0$ encode the digit zero (opposing arrows in $w_0$ negate each other), words $w_1\in L_1$ encode the digit one (arrows in the same direction in $w_1$ amplify each other) and in configurations of $X$ consecutive encodings of the same digit cannot occur. 

First let us note that $F(0^\Z)=0^\Z$ and $F(\orb{(ab)^\Z})=\orb{(ab)^\Z}$ for every $F\in\aut(X)$, because $0^\Z$ (resp. $(ab)^\Z$) are the only configurations (up to shift) of least period $1$ (resp. $2$) in $X$. Throughout this subsection let $e_\ell={}^\infty 0.1(ab)^\infty$ and $e_\arr={}^\infty(ab)1.0^\infty$.

\begin{lemma}If $F\in\aut(X)$, then $F(e_\ell)=\sigma^i(e_\ell)$ and $F(e_\arr)=\sigma^j(e_\arr)$ for some $i,j\in\Z$.\end{lemma}
\begin{proof}Let $F$ be a radius-$r$ reversible CA whose inverse also has radius $r$. We may assume without loss of generality (by composing $F$ with a suitable shift if necessary) that the rightmost occurrence of $1$ in $F(e_\ell)$ is at position $0$. We first claim that $F(e_\ell)$ does not contain any occurrence of words from $L_0\cup L_1$ (equivalently: $F(e_\ell)[-\infty,-1]={}^\infty0$). Otherwise assume without loss of generality that the leftmost such occurrence is from $L_0$. Let $x={}^\infty 01\yl\al 10^{3r+1}.0^\infty$ (i.e. $x$ contains an occurrence of a word from $L_0$). Its inverse image $F^{-1}(x)$ belongs to $X$ and thus also the gluing $F^{-1}(x)\glue e_\ell$ belongs to $X$ because the right infinite word $1(ab)^\infty$ in $e_\ell$ does not give additional constraints for the left side of the sequence. But then the configuration $F(F^{-1}(x)\glue e_\ell)$ contains two consecutive occurrences of words from $L_0$, contradicting the definition of $X$.

Now to prove that $F(e_\ell)=e_\ell$ it remains to show that $F(e_\ell)$ cannot contain any arrows, so we assume to the contrary that $F(e_\ell)$ contains one or two arrows. The possibility that $F(e_\ell)$ contains two arrows yields a contradiction by the same argument as in the previous paragraph (e.g. if $F(e_\ell)$ contains two opposing arrows, then glue $F^{-1}(x)\glue e_\ell$, in which case $F(F^{-1}(x)\glue e_\ell)$ contains two consecutive encodings of the digit $0$), so let us assume that $F(e_\ell)$ contains a single arrow (whose distance from the single $1$ in $F(e_\ell)$ is at most $r$). Let $e_\ell'={}^\infty 0.1(ab)^{2r+1}\yl(ab)^\infty$. Since $F$ is reversible, it follows that $F(e_\ell)\neq F(e_\ell')$ and in particular $F(e_\ell')$ contains two arrows. Now we can use the same argument as above to show that this is not possible, so we conclude that $F(e_\ell)=e_\ell$.

By symmetry $F(e_\arr)=\sigma^j(e_\arr)$ for some $j\in\Z$.
\end{proof}

For now, if $F,i,j$ are as in the previous lemma, we say that the \emph{intrinsic left (resp. right) shift} of $F$ is equal to $i$ (resp. equal to $j$). In the following let $x_\yl={}^\infty(ab).\yl(ab)^\infty$ and $x_\al={}^\infty(ab).\al(ab)^\infty$.

\begin{lemma}
If $F\in\aut(X)$ has intrinsic left shift $i$ (resp. intrinsic right shift $i$), then $F(x_\yl)\in\sigma^i(\{x_\yl,x_\al\})$ and $F(x_\al)\in\sigma^i(\{x_\al,x_\yl\})$. In particular the intrinsic right and left shift are equal.
\end{lemma}
\begin{proof}
Let $F$ be a radius-$r$ reversible CA whose inverse also has radius $r$ and assume without loss of generality (by composing $F$ with a suitable shift if necessary) that the intrinsic left shift is $i=0$, the case of the intrinsic right shift $i=0$ being symmetric. We prove the claim for $F(x_\yl)$, the other case being symmetric. We first claim that $F(x_\yl)\in\orb{x_\yl}\cup \orb{x_\al}$. Otherwise $F(x_\yl)$ contains at least two occurrences of arrows or at least one occurrence of $1$. Denoting $y={}^\infty0.1(ab)^{2r+1}\yl(ab)^\infty$, in both cases $F(y)[-\infty,2r+1]={}^\infty0.1(ab)^r$ by the previous lemma, and going further to the right in $F(y)$ there must be two occurrences of arrows after which there may be an occurrence of $1$. We will derive a contradiction in the case that these arrows point in opposing directions, after which it will be clear that a similar argument yields a contradiction the case that the arrows point in the same direction. Let $x={}^\infty 01\yl\al 10^{3r+1}.0^\infty$ (i.e. $x$ contains an occurrence of a word from $L_0$). The gluing $F^{-1}(x)\glue y$ belongs to $X$ because the right infinite word $1(ab)^{2r+1}\yl(ab)^\infty$ in $y$ does not give additional constraints for the left side of the sequence. But then the configuration $F(F^{-1}(x)\glue y)=x\glue F(y)$ contradicts the definition of $X$.

Now we prove that $F(x_\yl)\in\{x_\yl,x_\al\}$. Otherwise it holds that $F(x_\yl)\in\{\sigma^k(x_\yl),\sigma^k(x_\al)\}$ for $k\neq 0$ and we may assume without loss of generality that $0<k\leq r$ (by considering instead the CA $F^{-1}$ if necessary) and that $F(x_\yl)=\sigma^k(x_\yl)$ (by composing $F$ with the CA that only flips the direction of every arrow if necessary). Consider the point $x={}^\infty 0.1(ab)^{2r+1}\yl(ab)^\infty$. None of the configurations $F^t(x)$ $(t\in\N)$ contain an occurrence of a word from $L_0\cup L_1$ by the same argument as in the previous paragraph and as in the proof of the previous lemma. Similarly none of the $F^t(x)$ contain two arrows and the unique arrow in $F^t(x)$ points to the direction $\yl$. Since $F(x_\yl)=\sigma^k(x_\yl)$, it follows that for $t>0$ the distance between $1$ and $\yl$ in $F^t(x)$ is strictly smaller than in $x$ and in particular $F^t(x)\notin\orb{x}$. However, from $F(x_\yl)=\sigma^k(x_\yl)$ it also follows that in each $F^t(x)$ the distance between $1$ and $\yl$ is bounded, so $F^{t'}(x)=\sigma^m(F^{2t'}(x))$ for some $t'\in\Npos$, $m\in\Z$. Therefore $\sigma^m(F^{2t'}(x))$ has two distinct preimages under the map $\sigma^m\circ F^{t'}$ (they are $F^{t'}(x)$ and $\sigma^{-m}(x)$, for distinctness recall that $F^t(x)\notin\orb{x}$ for $t\in\Npos$), which contradicts the reversibility of $F$. 
\end{proof}

In the following we say that the \emph{intrinsic shift} of $F\in\aut(X)$ is equal to $i$ if $i$ is its intrinsic left (or equivalently right) shift. Next we will conclude that for any $F\in\aut(X)$ there are $0$-finite configurations with long contiguous segments of non-$0$ symbols on which $F$ cannot do anything nontrivial. In fact, this holds for every finitely generated subgroup of $\aut(X)$.

\begin{proposition}
For all $n\in\N$ let $x_n={}^\infty0.1(ab)^n\yl (ab)^n\yl (ab)^n 10^\infty$, $y_n={}^\infty0.1(ab)^n\al (ab)^n\al (ab)^n 10^\infty$ and $Z_n=\orb{\{x_n\}}\cup\orb{\{y_n\}}$. For every finitely generated $\grp{G}$ there is $N\in\N$ such that $F(Z_n)=Z_n$ for all $F\in\grp{G}$ and $n\geq N$.
\end{proposition}
\begin{proof}
Let $\{F_1,\dots,F_k\}\subseteq\aut(X)$ be a finite set that generates $\grp{G}$. Since the statement of the proposition concerns the shift-invariant sets $Z_n$, we may assume without loss of generality (by composing all the $F_i$ by suitable powers of the shift if necessary) that the intrinsic shift of every $F_i$ is equal to $0$. Fix a number $r\in\N$ such that all the $F_i$ are radius-$r$ CA whose inverses are also radius-$r$ CA. To prove the claim it is sufficient to show that $F_i(\{x_n,y_n\})=\{x_n,y_n\}$ for every $n\geq 2r+1$ and for every $1\leq i\leq k$. But this conclusion directly follows from the two previous lemmas.
\end{proof}

\subsection{Case study: $S$-gap shifts}\label{SpecSubSect}

One may ask how much it is possible to extend Theorem~\ref{sensitiveCA} to more general synchronizing subshifts. In this subsection we study a natural class of synchronizing subshifts known as $S$-gap shifts and we find out that at least in this class Theorem~\ref{sensitiveCA} cannot be generalized at all. A similar analysis on beta-shifts is presented in~\cite{Kop20beta}.

\begin{definition}
A subshift $X\subseteq A^\Z$ is a \emph{coded subshift} (generated by a language $L\subseteq A^+$) if $L(X)$ is the set of all subwords of elements of $L^*$.
\end{definition}

For nonempty $S\subseteq\N$, we define the $S$-gap shift $X_S\subseteq\digs_2^\Z$ as the coded subshift generated by $\{01^n\mid n\in S\}$. We may equate $S$ with its characteristic sequence and we write $S(i)=1$ if $i\in S$ and $S(i)=0$ if $i\notin S$ (for $i\in\N$).

Every $X_S$ is synchronizing, because $0$ is a synchronizing word. By Theorem~3.4 of~\cite{DJ12} an $S$-gap shift is sofic if and only if $S$ is eventually periodic. In particular $S$ is infinite and $1^\Z\in X_S$ whenever $X_S$ is not sofic.

Many $X_S$ satisfy an even stronger property.

\begin{definition}
We say that a subshift $X$ is a \emph{shift with specification} (with transition length $n\in\N$) if for every $u,v\in L(X)$ there is a $w\in L^n(X)$ such that $uwv\in L(X)$.
\end{definition}
All shifts with specification are synchronizing~\cite{Ber86}. 

By Example~3.4 of~\cite{Jung11} the subshift $X_S$ has the specification property if and only if the sequence $S\in\digs_2^\N$ does not contain arbitrarily long runs of zeroes between two ones and $\gcd\{n+1\mid n\in S\}=1$.

\begin{lemma}If $X_S$ is not sofic, then any $F\in\aut(X_S)$ has $1^\Z$ as a fixed point.\end{lemma}
\begin{proof}If $0\notin S$ then $1^\Z$ is the only fixed point of $X_S$ and we are done. Let therefore $0\in S$, assume to the contrary that $F(0^\Z)=1^\Z$ and $F(1^\Z)=0^\Z$ and consider the sequence of points $x_n={}^\infty 10^n1^\infty\in X_S$ ($n\in\N)$. Clearly the configurations $F(x_n)$ contain as subwords the words $01^n0$ for all sufficiently large $n$. But then $\N\setminus S$ would have to be finite, contradicting the assumption that $S$ is not eventually periodic.\end{proof}

For the rest of this section let $X_{S,\ell}=\{x\in X_S\mid x[0,\infty]=01^\infty\}$ and $X_{S,\arr}=\{x\in X_S\mid x[-\infty,0]={}^\infty10\}$. These sets are non-empty whenever $S$ is infinite.

\begin{lemma}Assume that $X_S$ is not sofic. For every $F\in\aut(X_S)$ there exists $i\in\Z$ such that $F(X_{S,\ell})\subseteq\sigma^i(X_{S,\ell})$ and $F(X_{S,\arr})\subseteq \sigma^i(X_{S,\arr})$.\end{lemma}
\begin{proof}Let $x_\arr\in X_{S,\arr}$ be arbitrary. Without loss of generality (by composing $F$ with a suitable power of the shift if necessary) we may assume that $F(x_\arr)\in X_{S,\arr}$. We will show that $F(X_{S,\ell})\subseteq X_{S,\ell}$. Let us therefore assume to the contrary and without loss of generality (by considering $F^{-1}$ instead of $F$ if necessary) that there exists $x_\ell\in X_{S,\ell}$ such that $F(x_\ell)\in \sigma^i(X_{S,\ell})$ for some $i>0$. Since $S$ is not eventually periodic, it follows that there are arbitrarily large $j\in\N$ such that $S(j)=1$ and $S(j+i)=0$. This is a contradiction, because for sufficiently large such $j$ it holds that $x=x_\ell[-\infty,-1]01^j.0 x_\arr[1,\infty]\in X_S$, but from $F(x_\ell)\in \sigma^i(X_{S,\ell})$ and $F(x_\arr)\in X_{S,\arr}$ it follows that $F(x)$ contains an occurrence of the forbidden pattern $01^{j+i}0$.

Because $F(X_{S,\ell})\subseteq X_{S,\ell}$, we can use the argument of the previous paragraph to show that $F(X_{S,\arr})\subseteq X_{S,\arr}$.
\end{proof}

If $F$ and $i$ are as in the previous lemma, we say that the \emph{intrinsic shift} of $F$ is equal to $i$. If $i=0$, we say that $F$ is \emph{shiftless}.

\begin{corollary}Assume that $X_S$ is not sofic. Let $F\in\aut(X_S)$ be a shiftless radius-$r$ automorphism and let $f:\digs_2^{2r+1}\to \digs_2$ be a local rule that defines $F$. For any word $w\in \digs_2^r$ the following hold: $f(w01^r)=f(1^r0w)=0$, $f(w1^{r+1})=1$ and $f(1^{r+1}w)=1$ (whenever all the words involved are from $L(X_S)$).\end{corollary}

\begin{corollary}\label{gapinert}Assume that $X_S$ is not sofic. Let $F\in\aut(X_S)$ be a shiftless radius-$r$ automorphism whose inverse is also a radius-$r$ automorphism. If $x\in X_S$, $i\in\Z$ and $n\geq 2r$, then $01^n0$ occurs in $x$ at position $i$ if and only if it occurs in $F(x)$ at position $i$.\end{corollary}

Now we can show that Theorem~\ref{sensitiveCA} does not extend to general synchronizing shifts and not even to general shifts with specification.

\begin{theorem}
Assume that $X_S$ is not sofic. Then every reversible cellular automaton $F:X_S\to X_S$ has an almost equicontinuous direction.
\end{theorem}
\begin{proof}
Assume that $F$ has intrinsic shift $i$. Then $F'=\sigma^{-i}\circ F$ is shiftless. Let $r$ be a radius for both $F'$ and its inverse and choose an arbitrary $n\in S$ such that $n\geq 2r$. By the previous corollary the word $01^n0$ is blocking for $F'$ so by Proposition~\ref{equiblock} $F'$ is almost equicontinuous. Then $-i$ is an almost equicontinuous direction for $F$.
\end{proof}

Corollary~\ref{gapinert} can also be used to show that Theorem~\ref{transRyan} (Finitary Ryan's theorem) does not extend to shifts with specification.

\begin{theorem}\label{nogapFin}If $X_S$ is not sofic, then $k(X_S)=\infty$.\end{theorem}
\begin{proof}We argue similarly as in the proof of Proposition~\ref{noCompact}. In any case $k(X_S)\neq\bot$ by Theorem~\ref{synchRyan}. To see that $k(X_S)=\infty$, assume to the contrary that $R\subseteq\aut(X_S)$ is a set of cardinality of $n\in\N$ such that $C(R)=\left<\sigma\right>$. By composing the elements of $R$ by suitable powers of the shift we may assume without loss of generality that all the elements of $R$ are shiftless. Fix a number $r\in\Npos$ such that all elements of $R$ are radius-$r$ automorphisms whose inverses are also radius-$r$ automorphisms.

Let $n_1<n_2<n_3\in S$ be three distinct numbers such that $n_i\geq 2r$. Let $w_i=1^{n_i}$ and let $H\in\aut(X_S)$ be the automorphism which given a point $x\in X_S$ replaces every occurrence of the pattern
\[0 w_3 0 w_1 0 w_2 0 \qquad \mbox{by} \qquad 0 w_3 0 w_2 0 w_1 0\]
and vice versa (it exists by Lemma~\ref{markerauto} with the choice $u=0$). In light of Corollary \ref{gapinert} it is evident that the elements of $R$ cannot remove or add occurrences of the patterns defined above, so $H$ commutes with every element of $R$, a contradiction.
\end{proof}

Combining this with Theorem~\ref{transRyan} yields the following seemingly strong (but perhaps not surprising) corollary.

\begin{corollary}
If $X_S$ is not sofic then $\aut(X_S)\not\simeq\aut(Z)$ for every transitive sofic $Z$.
\end{corollary}

\section{Conclusions}

We conclude with some speculations. We guess that whenever $X$ is a transitive subshift for which $\aut(X)$ is ``large'' as an abstract group, then $k(X)<\infty$ implies that $\aut(X)$ contains a reversible CA without almost equicontinuous directions. This would be interesting because it would connect a group theoretical property of $\aut(X)$ to the possible CA dynamics on the subshift $X$.

The group $\aut(X)$ is large at least when $X$ is an infinite synchronizing subshift in the sense that it contains an isomorphic copy of the free product of all finite groups~\cite{Fie96}. If $X$ is a nontrivial mixing sofic shift, then by Theorems~\ref{sensitiveCA} and~\ref{transRyan} $\aut(X)$ contains a CA without almost equicontinuous directions and $k(X)=2$. On the other hand, in the previous subsection we saw examples of subshifts $X$ with the specification property such that every $F\in\aut(X)$ has a direction that admits blocking words and we used the existence of blocking words to prove that $k(X)=\infty$.

The assumption of largeness of $\aut(X)$ is necessary. By~\cite{Fie96} for any finite group $\grp{G}$ there is a coded subshift $X$ such that $\aut(X)\simeq \Z\oplus\grp{G}$, where the part $\Z$ corresponds to the shift maps. Then $k(X)=0$ whenever $C_{\grp{G}}(\grp{G})=\{1_{\grp{G}}\}$ but every element of $\aut(X)$ has an almost equicontinuous direction.

\begin{problem}\label{probSynchEqui}
Is $k(X)=\infty$ for every infinite synchronizing subshift $X$ such that every $F\in\aut(X)$ admits an almost equicontinuous direction?
\end{problem}

We note that there are synchronizing non-sofic subshifts that admit CA with only sensitive directions. For example, whenever $X$ is synchronizing and non-sofic, then so is also $Y=X\times X$ and the CA $F:Y\to Y$ defined by $F(x_1,x_2)=(\sigma(x_1),\sigma^{-1}(x_2))$ for $x_1,x_2\in X$ has only sensitive directions. In the light of examples such as this, it is not clear what kind of an answer one should expect to the following problem. 

\begin{problem}\label{probTransEqui}
Characterize the transitive non-sofic subshifts that admit reversible CA with only sensitive directions.
\end{problem}

We guess that $k(Y)=\infty$ at least when $Y=X_S\times X_S$ for some synchronizing non-sofic $S$-gap shift $X_S$, which would mean that the existence of reversible CA with only sensitive directions is not sufficient to prove a finitary Ryan's theorem for general synchronizing shifts.

\begin{problem}\label{probSynchRyan}
Is $k(X)=\infty$ for every non-sofic synchronizing subshift $X$?
\end{problem}

We also ask whether the existence of a reversible CA $F:X\to X$ with only sensitive directions on a subshift $X$ has a simple dynamical characterization based on $X$ or a simple combinatorial characterization based on the language $L(X)$ or the syntactic monoid $\syn_X$.

\section*{Acknowledgements}
I thank Mike Boyle for suggesting that I present the construction of Section~\ref{glidertransSect} so that diffusion can happen against any chosen periodic background that contains a synchronizing word (instead of a specially selected background as in~\cite{KopDiss}). I also thank him for suggesting that I include Remark~\ref{BoyleRmrk}. The work was supported by the Academy of Finland grant 296018.

\bibliographystyle{plain}
\bibliography{mybib}{}

\end{document}